\documentclass[11pt,reqno]{amsart}
\usepackage[dvipsnames]{xcolor}
\usepackage{amsmath,amssymb,amsthm,amsfonts,enumerate,tikz,bm}
\usepackage{mathrsfs}
\usepackage{enumitem}
\usetikzlibrary{matrix,arrows,positioning,calc}
\usepackage{bbm}
\usepackage{bm}
\usepackage{mathrsfs}
\usepackage{latexsym}
\usepackage[all]{xy}
\usepackage[bookmarksnumbered,colorlinks]{hyperref}
\usepackage{graphicx}
\usepackage{CJK,CJKnumb}
\usepackage{extarrows}

\date{}
\renewcommand{\uppercasenonmath}[1]{}

\makeatletter
   \DeclareMathSizes{\@xipt}{\@xipt}{7}{2}
\makeatother
\newtheorem*{thm*}{Main Theorem}
\newtheorem{thm}{Theorem}[section]
\newtheorem{cor}[thm]{Corollary}
\newtheorem*{cor*}{Corollary}
\newtheorem{lem}[thm]{Lemma}
\newtheorem*{lem*}{Lemma}
\newtheorem{prop}[thm]{Proposition}
\newtheorem*{prop*}{Proposition}

\theoremstyle{definition}

\newtheorem{remark}[thm]{Remark}
\newtheorem{exa}[thm]{Example}

\newtheorem{df}[thm]{Definition}
\newtheorem*{df*}{Definition}

\newtheorem*{conj*}{Conjecture}

\newcommand{\pf}{\noindent\begin {proof}}
\newcommand{\epf}{\end{proof}}
\newtheorem*{ack*}{ACKNOWLEDGEMENTS}

\newcommand{\bc}{\begin{center}}
\newcommand{\ec}{\end{center}}

\def\Qcoh{\mathfrak{Qcoh}}
\newcommand{\R}{\mathcal{O}_X}
\newcommand{\Hom}{\mathsf{Hom}}
\newcommand{\Ext}{\mathsf{Ext}}

\newcommand{\mcC}{\mathcal{C}}
\newcommand{\Rep}{\mathrm{Rep}}
\newcommand{\mathcolon}{\colon\,}

\begin{document}
\title{On some applications of balanced pairs and their relation with cotorsion triplets}
\author{Sergio Estrada}
\address[S. Estrada]{
Departamento de Matem\'aticas, Universidad de Murcia,
Campus de Espinardo, Espinardo (Murcia) 30100, Spain}
\email{sestrada@um.es}%
%\author{Jiangsheng Hu}
%\address[J. Hu]{Department of Mathematics, Jiangsu University of Technology, Changzhou 213001, China}
%\email{jiangshenghu@jsut.edu.cn}
\author{Marco A. P\'erez}
\address[M. A. P\'erez]{Instituto de Matem\'atica y Estad\'{\i}stica ``Prof. Ing. Rafael Laguardia". Universidad de la Rep\'ublica. Montevideo 11300, Uruguay}
\email{mperez@fing.edu.uy}
\author{Haiyan Zhu}
\address[H. Zhu]{College of Science, Zhejiang University of Technology, Hangzhou 310023, China}
\email{hyzhu@zjut.edu.cn}


\begin{abstract}
Balanced pairs appear naturally in the realm of Relative Homological Algebra associated to the balance of right derived functors of the $\Hom$ functor. A natural source to get such pairs is by means of cotorsion triplets. In this paper we study the connection between balanced pairs and cotorsion triplets by using recent quiver representation techniques. In doing so, we find a new characterization of abelian categories having enough projectives and injectives in terms of the existence of complete hereditary cotorsion triplets. We also give a short proof of the lack of balance for derived functors of $\Hom$ computed by using flat resolutions which extends the one showed by Enochs in the commutative case.
\end{abstract}

\maketitle
\pagestyle{myheadings}
\markboth{\rightline {\scriptsize S. Estrada, M. A. P\'{e}rez and H. Zhu}}
         {\leftline{\scriptsize On some applications of balanced pairs and their relation with cotorsion triplets}}

%%%%%%%%%%%%%%%%%%%%%%%%%%%%%%%%%%%%%%%%%%%%%%%%%%
%%%%%%%%%%%%%%%%%%%%%%%%%%%%%%%%%%%%%%%%%%%%%%%%%%
%%%%%%%%%%%%%%%%%%%%%INTRODUCTION%%%%%%%%%%%%%%%%%%%%
%%%%%%%%%%%%%%%%%%%%%%%%%%%%%%%%%%%%%%%%%%%%%%%%%%
%%%%%%%%%%%%%%%%%%%%%%%%%%%%%%%%%%%%%%%%%%%%%%%%%%

\section{Introduction}
Let $\mathcal C$ be an abelian category and $\mathcal F$ be a precovering class. This means that for each object $M\in \mathcal C$ there exists a (not necessarily exact) complex
\[
\cdots \to F_1\to F_0\to M\to 0,
\]
usually called an $\mathcal F$-resolution of $M$, where $F_i\in \mathcal F$ for every $i\geq 0$, which is exact after applying the functor $\Hom_{\mathcal C}(F,-)$ for each $F\in \mathcal F$. The corresponding deleted complex is unique up to homotopy, so we can compute right derived functors of $\Hom$, denoted by $\mathcal{F}\textrm{-}\Ext^n$. In many cases there is ``balance'' in the computation of such functors, meaning that there exists a preenveloping class $\mathcal L$ such that $\mathcal{F}\textrm{-}\Ext^n(M,N)$ can be also obtained from the right derived functors $\mathcal{L}\textrm{-}\Ext^n$ computed from of a coresolution of $N$,
\[
0\to N\to L_0\to L_1\to \cdots,
\]
where $L_i\in \mathcal L$ for every $i\geq 0$. This phenomenon can be summarized by saying that the pair $(\mathcal{F,L})$ is a balanced pair (in the sense of Chen \cite{Chen}) or equivalently that the functor $\Hom$ is right balanced by $\mathcal F\times\mathcal L$ (see Enochs and Jenda \cite[Section 8.2]{EJ}).

Thus balanced pairs have gained attention in the last years in the context of Relative Homological Algebra (see for instance \cite{Chen, Enochs, EJ, EJTX, EHHS}). Our goal in this paper is to deepen in the relation between balanced and cotorsion pairs or, to be more precise, between balanced pairs and complete and hereditary cotorsion triplets. Recall that a triplet $(\mathcal{F,G,L})$ is called a cotorsion triplet provided that $(\mathcal{F,G})$ and $(\mathcal{G,L})$ are cotorsion pairs. The reader can have in mind the trival cotorsion triplet $(\mathsf{Proj}(R),\mathsf{Mod}(R),\mathsf{Inj}(R))$ in the category $\mathsf{Mod}(R)$ of left $R$-modules (where $\mathsf{Proj}(R)$ and $\mathsf{Inj}(R)$ denote the classes of projective and injective left $R$-modules respectively) as the canonical example of a complete and hereditary cotorsion triplet. But there are many other instances of such triplets occuring in practice (see Example \ref{ex:triplets}).

Complete hereditary cotorsion triplets are defined in Definition \ref{df:triplets}. They are a natural source to providing with balanced pairs. In short, a complete hereditary cotorsion triplet $(\mathcal{F,G,L})$ in an abelian category with enough projectives and injectives gives a balanced pair $(\mathcal{F,L})$ (see Enochs, Jenda, Torrecillas and Xu \cite[Theorem 4.1]{EJTX}).

Thus, it seems natural to wonder about the converse of this result. This appears explicitly as an open problem in \cite[Open Problems]{EJTX}.

\medskip\par\noindent
{\bf Question}: Find conditions for a balanced pair $(\mathcal{F,L})$ to induce a complete hereditary cotorsion triplet $(\mathcal{F,G,L})$.
\medskip\par\noindent

One of our motivations in this paper is to shed any light on this question.  We show in Proposition~\ref{prop:from_balance_to_triplets} that a balanced pair $(\mathcal{F,L})$ in an abelian category with enough projectives and injectives gives rise to a complete hereditary cotorsion triplet $(\mathcal{F,G,L})$, provided that the class $\mathcal{F}$ is resolving and special precovering, $\mathcal{L}$ is  coresolving and special preenveloping, and that $\mathcal{F}$ and $\mathcal{G}$ satisfy the relations $\mathcal{F} \cap \mathcal{F}^\perp \subseteq {}^\perp\mathcal{L}$ and ${}^\perp\mathcal{L} \cap \mathcal{L} \subseteq \mathcal{F}^\perp$.

Let us point out that we cannot expect to get such triplet from \emph{any} balanced pair. For instance, given any ring $R$ with identity, the pair $(\mathsf{Mod}(R),\mathsf{Mod}(R))$ is trivially a balanced pair, but the triplet $(\mathsf{Mod}(R),\mathcal G,\mathsf{Mod}(R))$ is complete if and only if $R$ is quasi-Frobenius.

However, we give a partial answer to the previous question by working with the abelian category $\Rep(Q,\mathcal C)$ of $\mathcal C$-valued representations over a non-discrete quiver $Q$. The precise formulation of our result is the following. The proof is in Corollary~\ref{cor 2}.

\medskip\par\noindent
{\bf Theorem.} If $(\mathcal{F},\mathcal{H})$ and $(\mathcal{G},\mathcal{L})$ are complete hereditary cotorsion pairs in $\mathsf{Mod}(R)$, then the following are equivalent:
\begin{itemize}
\item [(a)]$\mathcal{H}=\mathcal{G}$ (that is, $(\mathcal{F,G,L})$ is a complete and hereditary cotorsion triplet in $\mathsf{Mod}(R)$).
\item [(c)] $(\Phi(\mathcal{F}),\Psi(\mathcal{L}))$ is a balanced pair in $\Rep(Q,\mathsf{Mod}(R))$ for some non-discrete left and right rooted quiver $Q$.
\end{itemize}

The classes $\Phi(\mathcal{F})$ and $\Psi(\mathcal{L})$ are defined by Holm and J\o rgensen in \cite{HJ}. We recall in Section~\ref{sec:quiver} their definition.

Notice that one easy example of left and right rooted quiver is the 1-arrow quiver $Q:\bullet\to \bullet$, and so in this case $\Rep(Q,\mathsf{Mod}(R))$ is nothing but the category $\mathsf{Mor}(R)$ of morphisms of $R$-modules. But there are many other (possibly infinite) quivers satisfying this condtion. In short, the previous Theorem assures that in order to look for conditions to get an equivalence between balanced pairs and cotorsion triplets, we need to move to a ``bigger" category. This result allows to characterize quasi-Frobeinus rings (Corollary \ref{cor:qf}) in terms of the so-called \emph{monomorphism category} and \emph{epimorphism category} as considered by Li, Luo and Zhang in \cite{LiZ,LuoZ}. And also we recover and extend the recent characterization of virtually Gorenstein rings given by Zareh-Khoshchehreh, Asgharzadeh and Divaani-Aazar in \cite[Theorem 3.10]{ZAD}.

While studying cotorsion triplets, we found the following interesting result of independent interest (see Theorem~\ref{theorem.cotorsion.triplets}).

\medskip\par\noindent
{\bf Theorem.} An abelian category $\mathcal C$ has enough projectives and injectives if, and only if, there exists a hereditary and complete cotorsion triplet in $\mathcal C$.

\medskip\par\noindent

This theorem allows us to present a slightly stronger version of the mentioned result \cite[Theorem 4.1]{EJTX} by Enochs, Jenda, Torrecillas and Xu (see Proposition~\ref{prop.relation.balance.triplets}).

Finally, we give in Theorem \ref{theorem.noflatbalance} a short and categorical proof about the lack of balance with respect to the class of flat modules over a left Noetherian non-perfect ring. Our method is different from the one used by Enochs in \cite[Theorem 4.1]{Enochs} for the commutative case. As a consequence we give a negative answer in Corollary 4.2 to the question 6 posted in \cite[Section 6]{Enochs}. Namely, we show in Corollary~\ref{cor:balance_scheme} that there is no balance for the class of flat quasi-coherent modules on a Noetherian and semi-separated scheme.

%%%%%%%%%%%%%%%%%%%%%%%%%%%%%%%%%%%%%%%%%%%%%%%%%%
%%%%%%%%%%%%%%%%%%%%%%%%%%%%%%%%%%%%%%%%%%%%%%%%%%
%%%%%%%%%%%%%%%%%%%%%PRELIMINARIES%%%%%%%%%%%%%%%%%%%%
%%%%%%%%%%%%%%%%%%%%%%%%%%%%%%%%%%%%%%%%%%%%%%%%%%
%%%%%%%%%%%%%%%%%%%%%%%%%%%%%%%%%%%%%%%%%%%%%%%%%%

\section{Preliminaries}

Throughout, $\mathcal{C}$ will denote an abelian category.

%%%%%%%%%%%%%%%%%%%%%%%%%%%%%%%%%%%%%%%%%%%%%%%%%%
%%%%%%%%%%%%%%%%%%%%%%%%%%%%%%%%%%%%%%%%%%%%%%%%%%

\subsection*{Cotorsion pairs in abelian categories}

Two classes of objects $\mathcal{X}$ and $\mathcal{Y}$ in $\mathcal{C}$ form a \emph{cotorsion pair} $(\mathcal{Y,X})$ if the following two equalities hold:
\begin{align*}
\mathcal{Y} & = {}^{\perp_1}\mathcal{X} := \{ C \in \mcC \mbox{ : } \Ext^1_{\mcC}(C,X) = 0 \mbox{ for every } X \in \mathcal{X} \}, \\
\mathcal{X} & = \mathcal{Y}^{\perp_1} := \{ D \in \mcC \mbox{ : } \Ext^1_{\mcC}(Y,D) = 0 \mbox{ for every } Y \in \mathcal{Y} \}.
\end{align*}
Since $\mathcal{C}$ does not necessarily have enough projectives and/or injectives, the extension groups $\Ext^i_{\mcC}(A,B)$ are defined via its Yoneda description as certain equivalent classes of $i$-fold extensions.

A cotorsion pair $(\mathcal{Y,X})$ in $\mathcal{C}$ is called:
\begin{enumerate}
\item \label{def:complete} \emph{Complete} if for every object $C \in \mathcal{C}$ there exist short exact sequences
\[
0 \to X \to Y \to C \to 0 \mbox{ \ and \ } 0 \to C \to X' \to Y' \to 0
\]
with $Y, Y' \in \mathcal{Y}$ and $X, X' \in \mathcal{X}$.

\item \label{def:hereditary} \emph{Hereditary} if $\Ext^i_{\mcC}(Y,X) = 0$ for every $Y \in \mathcal{Y}$ and $X \in \mathcal{X}$, and $i > 0$.
\end{enumerate}

Recall that a class $\mathcal{Y}$ of objects in $\mathcal{C}$ is \emph{resolving} if $\mathcal{Y}$ is closed under extensions and under kernels of epimorphisms with domain and codomain in $\mathcal{Y}$, and if $\mathcal{Y}$ contains the class of projective objects in $\mathcal{C}$. Dually, one has the notion of \emph{coresolving class}. We say that a cotorsion pair $(\mathcal{Y,X})$ in $\mathcal{C}$ is \emph{quasi-hereditary} if $\mathcal{Y}$ is resolving and $\mathcal{X}$ is coresolving. In some references, quasi-hereditary cotorsion pairs are called hereditary, but the two notions are not the same in general. Indeed, the condition defining hereditary cotorsion pairs in \eqref{def:hereditary} above is stronger than asking $\mathcal{Y}$ and $\mathcal{X}$ to be resolving and coresolving, respectively. This can be appreciated in the following result, whose proof is well known.

\begin{prop}
Every hereditary cotorsion pair in $\mathcal{C}$ is quasi-hereditary. If in addition, $\mathcal{C}$ has enough projectives and injectives, then every quasi-hereditary cotorsion pair in $\mathcal{C}$ is hereditary.
\end{prop}

If $(\mathcal{Y,X})$ is a hereditary cotorsion pair in $\mathcal{C}$, we actually have that:
\begin{align*}
\mathcal{Y} & = {}^\perp\mathcal{X} := \{ C \in \mathcal{C} \mbox{ : } \Ext^i_{\mathcal{C}}(C,X) = 0 \mbox{ for every $X \in \mathcal{X}$ and $i > 0$} \}, \\
\mathcal{X} & = \mathcal{Y}^\perp := \{ D \in \mathcal{C} \mbox{ : } \Ext^j_{\mathcal{C}}(Y,D) = 0 \mbox{ for every $Y \in \mathcal{Y}$ and $j > 0$} \}.
\end{align*}

%%%%%%%%%%%%%%%%%%%%%%%%%%%%%%%%%%%%%%%%%%%%%%%%%%
%%%%%%%%%%%%%%%%%%%%%%%%%%%%%%%%%%%%%%%%%%%%%%%%%%

\subsection*{Precovering and preenveloping classess}

Let $\mathcal F$ be a class of objects in $\mathcal C$. A morphism ${\phi} \mathcolon {F} \to {M}$ in $\mathcal C$ is called an \emph{$\mathcal F$-precover} of $M$ if $F \in \mathcal F$ and
\[
\Hom_{\mathcal C}(F',F) \to \Hom_{\mathcal C} (F',M) \to 0
\]
is an exact sequence of abelian groups for every object $F' \in \mathcal F$. Further, if ${\phi}\mathcolon {F}\to{M}$ is an $\mathcal{F}$-precover and $\ker(\phi)\in \mathcal F^{\perp_1}$ then $\phi$ is called a \emph{special $\mathcal F$-precover}. If every object in $\mathcal C$ has a (special) $\mathcal F$-precover, then the class $\mathcal F$ is called \emph{(special) precovering}.

The dual notions are \emph{(special) preenvelope} and \emph{(special) preenveloping classes}. It is easy to observe that, if $(\mathcal {Y,X})$ is a complete cotorsion pair in $\mathcal{C}$, then $\mathcal Y$ is special precovering and $\mathcal X$ is special preenveloping.

By using a standard argument (known as Salce's trick) we get the following lemma.

\begin{lem}\label{lema.complete.hered.sprecovering}
Suppose that $\mathcal C$ has enough projectives and injectives. Then, the following hold:
\begin{enumerate}
\item Let $\mathcal F$ be a special precovering class in $\mathcal{C}$ which is also resolving and closed under direct summands. Then, $(\mathcal F,\mathcal F^{\perp})$ is a complete hereditary cotorsion pair in $\mathcal{C}$.

\item Let $\mathcal L$ be a special preenveloping class in $\mathcal{C}$ which is also coresolving and closed under direct summands. Then, $({^{\perp}\mathcal L},\mathcal L)$ is a complete hereditary cotorsion pair in $\mathcal{C}$.
\end{enumerate}
\end{lem}

%%%%%%%%%%%%%%%%%%%%%%%%%%%%%%%%%%%%%%%%%%%%%%%%%%
%%%%%%%%%%%%%%%%%%%%%%%%%%%%%%%%%%%%%%%%%%%%%%%%%%

\subsection*{Resolutions and coresolutions}

Let $\mathcal X$ be a class of objects in $\mathcal C$ and $M$ an object in $\mathcal{C}$. An \emph{$\mathcal X$-resolution} $X_\bullet\to M$ of $M$ is a (not necessarily exact) complex
\[
\cdots\to X_1\to X_0\to M\to 0,
\]
with each $X_i\in \mathcal X$, which is exact when applying the functor $\Hom_{\mathcal C}(X,-)$, for every $X\in \mathcal X$. In this case, we will say that the complex $X_\bullet \to M$ is \emph{$\Hom_{\mathcal C}(\mathcal X,-)$-acyclic}. Dually, we have the notion of \emph{$\mathcal X$-coresolution} $M\to X^\bullet$ of $M$.

If $\mathcal X$ is precovering (respectively, $\mathcal X$ is preenveloping) it is easy to see that every $M$ in $\mathcal C$ has an $\mathcal X$-resolution (respectively, an $\mathcal X$-coresolution). See, for instance, Enochs and Jenda \cite[Proposition 8.1.3]{EJ}.

%%%%%%%%%%%%%%%%%%%%%%%%%%%%%%%%%%%%%%%%%%%%%%%%%%
%%%%%%%%%%%%%%%%%%%%%%%%%%%%%%%%%%%%%%%%%%%%%%%%%%
%%%%%%%%%%%%%%%%%%%%BALANCED%PAIRS%%%%%%%%%%%%%%%%%%%%
%%%%%%%%%%%%%%%%%%%%%%%%%%%%%%%%%%%%%%%%%%%%%%%%%%
%%%%%%%%%%%%%%%%%%%%%%%%%%%%%%%%%%%%%%%%%%%%%%%%%%

\section{Relation between balanced pairs and cotorsion pairs}

Let us begin this section recalling the definition of balanced pairs in abelian categories.

\begin{df}\label{def1}
A pair $(\mathcal{F},\mathcal{L})$ of additive subcategories in $\mathcal{C}$ is called a \textbf{balanced
pair} if the following conditions are satisfied:
\begin{itemize}
\item[(BP0)] $\mathcal{F}$ is precovering and $\mathcal{L}$ is preenveloping.

\item[(BP1)] For each object $M \in \mathcal{C}$, there is an $\mathcal{F}$-resolution $F_\bullet \rightarrow M$ which is ${\Hom}_{\mathcal{C}}(-,\mathcal{L})$-acyclic.

\item[(BP2)] For each object $M \in \mathcal{C}$, there is a $\mathcal{L}$-coresolution $M\rightarrow L^\bullet$ which is ${\Hom}_{\mathcal{C}}(\mathcal{F},-)$-acyclic.
\end{itemize}
A balanced pair is called \textbf{admissible} provided that each $\mathcal F$-precover is an epimorphism and each $\mathcal L$-preenvelope is a monomorphism.
\end{df}

We have the following useful characterization of balanced pairs:

\begin{lem}\label{lem 3}
Let $\mathcal{F}$ and $\mathcal{L}$ be a precovering and a preenveloping class in $\mathcal{C}$, respectively. Then, the following conditions are equivalent.
\begin{itemize}
\item[(a)] The pair $(\mathcal{F},\mathcal{L})$ is balanced.

\item[(b)] Each ${\Hom}_{\mathcal{C}}(\mathcal{F},-)$-acyclic and left exact sequence in $\mathcal{C}$ is also ${\Hom}_{\mathcal{C}}(-,\mathcal{L})$-acyclic, and each ${\Hom}_{\mathcal{C}}(-,\mathcal{L})$-acyclic and right exact sequence in $\mathcal{C}$ is also ${\Hom}_{\mathcal{C}}(\mathcal{F},-)$-acyclic.

\item[(c)] For each object $M \in \mathcal{C}$, there is a left exact sequence $0\rightarrow K\rightarrow F\rightarrow M\rightarrow0$ and a right exact sequence $0\rightarrow M\rightarrow L\rightarrow C\rightarrow0$, which are both ${\Hom}_{\mathcal{C}}(\mathcal{F},-)$-acyclic and ${\Hom}_{\mathcal{C}}(-,\mathcal{L})$-acyclic, where $F \in \mathcal{F}$ and $L \in \mathcal{L}$.
\end{itemize}
\end{lem}

\pf
The implication (a) $\Rightarrow$ (b) follows from Chen's \cite[Proposition 2.2]{Chen}, while (b) $\Rightarrow$ (c) is clear. Let us finish the proof showing (c) $\Rightarrow$ (a). By the assumption (c), for each object $M \in \mathcal{C}$ there is a left exact sequence
\[
0 \rightarrow K_0 \rightarrow F_0 \rightarrow M \rightarrow 0
\]
in $\mathcal{C}$ with $F_0 \in \mathcal{F}$ which is ${\Hom}_{\mathcal{C}}(\mathcal{F},-)$-acyclic and ${\Hom}_{\mathcal{C}}(-,\mathcal{L})$-acyclic. Now, by applying (c) again to the object $K_0$ we get a left exact sequence
\[
0 \rightarrow K_1 \rightarrow F_1 \rightarrow K_0 \rightarrow 0
\]
with $F_1 \in \mathcal{F}$ which is ${\Hom}_{\mathcal{C}}(\mathcal{F},-)$-acyclic and ${\Hom}_{\mathcal{C}}(-,\mathcal{L})$-acyclic. Continuing this process, we obtain an $\mathcal{F}$-resolution $F_\bullet \rightarrow M$ which is ${\Hom}_{\mathcal{C}}(-,\mathcal{L})$-acyclic. The construction of a ${\Hom}_{\mathcal{C}}(\mathcal{F},-)$-acyclic $\mathcal{L}$-coresolution of $M$ is dual. Hence, (a) follows.
\epf

%%%%%%%%%%%%%%%%%%%%%%%%%%%%%%%%%%%%%%%%%%
%%%%%%%%%%%%%%%%%%%%%%%%%%%%%%%%%%%%%%%%%%

\subsection*{Balanced pairs vs. cotorsion pairs}

As a first consequence of the previous result, we can infer the following relation between cotorsion pairs and balanced pairs. From now on, we will denote by $\mathsf{Proj}(\mathcal{C})$ and $\mathsf{Inj}(\mathcal{C})$ the classes of projective and injective objects of $\mathcal{C}$, respectively.

\begin{prop}
Let $(\mathcal{F,H})$ and $(\mathcal{G,L})$ be cotorsion pairs in $\mathcal{C}$ such that the pair $(\mathcal{F,L})$ is balanced. Then, $\mathcal{F}\cap\mathcal{G}=\mathsf{Proj}(\mathcal{C})$ and $\mathcal{H}\cap\mathcal{L}= \mathsf{Inj}(\mathcal{C})$.
\end{prop}

\pf
Let us only prove the equality $\mathcal{F}\cap\mathcal{G}=\mathsf{Proj}(\mathcal{C})$. The corresponding statement with injectives follows in a dual manner. Since $(\mathcal{F,H})$ and $(\mathcal{G,L})$ are cotorsion pairs, the containment $\mathsf{Proj}(\mathcal{C})\subseteq \mathcal{F}\cap\mathcal{G}$ always holds. Conversely, let $H\in \mathcal{F}\cap\mathcal{G}$ and $C\in \mathcal{C}$ be an arbitrary object. Let us consider an element in $\Ext^1_{\mathcal{C}}(H,C)$ represented by an exact sequence
\begin{equation}\label{eq1}
0\to C\to D\to H\to 0.
\end{equation}
Since $H\in \mathcal G$, the sequence \eqref{eq1} is $\Hom_{\mathcal{C}}(-,\mathcal{L})$-acyclic. But then by Lemma \ref{lem 3}, we have that this sequence is also $\Hom_{\mathcal{C}}(\mathcal{F},-)$-acyclic. This in turn implies that \eqref{eq1} splits, since $H\in \mathcal{F}$. Finally, being $C$ arbitrary, we conclude that $H$ is projective.
\epf

%%%%%%%%%%%%%%%%%%%%%%%%%%%%%%%%%%%%%%%%%%
%%%%%%%%%%%%%%%%%%%%%%%%%%%%%%%%%%%%%%%%%%

\subsection*{Uniqueness of balanced pairs}

Given a preenveloping class $\mathcal L$ in $\mathcal{C}$, there might be two different classes $\mathcal F_1$ and $\mathcal F_2$ such that $(\mathcal{F}_1,\mathcal L)$ and $(\mathcal{F}_2,\mathcal L)$ are balanced pairs. For instance, take the category $\mathcal{C} = \mathsf{Mod}(R)$ of left $R$-modules and $\mathcal{L}$ the class of all injective left $R$-modules. Then, we have two balanced pairs $(\mathcal{F}_1,\mathcal{L})$ and $(\mathcal{F}_2,\mathcal{L})$, where $\mathcal{F}_1$ is the class of all free left $R$-modules and $\mathcal F_2$ consists of all projective left $R$-modules. In this example we notice that $\mathsf{Smd}(\mathcal{F}_1)=\mathsf{Smd}(\mathcal{F}_2)$ (where the notation $\mathsf{Smd}(\mathcal{F})$ stands for the class of direct summands of objects in $\mathcal F$). The second consequence of Lemma~\ref{lem 3} shows that this sort of uniqueness property holds for any admissible balanced pair.

\begin{prop}
If $(\mathcal F_1,\mathcal L)$ and $(\mathcal F_2,\mathcal L)$ are two admissible balanced pairs in $\mathcal{C}$, then the equality $\mathsf{Smd}(\mathcal{F}_1)=\mathsf{Smd}(\mathcal{F}_2)$ holds. Dually, if $(\mathcal{F},\mathcal{L}_1)$ and $(\mathcal{F},\mathcal{L}_2)$ are two admissible balanced pairs in $\mathcal{C}$, then $\mathsf{Smd}(\mathcal{L}_1) = \mathsf{Smd}(\mathcal{L}_2)$.
\end{prop}

\pf
Let us see that $\mathsf{Smd}(\mathcal F_1)\subseteq \mathsf{Smd}(\mathcal F_2)$. The other inclusion follows by the same argument. It is easy to observe that it suffices to show $\mathcal{F}_1\subseteq \mathsf{Smd}(\mathcal F_2)$. First, note that since $\mathcal F_2$ is a precovering class in $\mathcal{C}$, for any $F_1\in \mathcal F_1$ we have a $\Hom_{\mathcal C}(\mathcal{F}_2, -)$-acyclic left exact sequence
\begin{equation}\label{equat}
0 \to K \to F_2 \to F_1 \to 0.
\end{equation}
in $\mathcal{C}$ with $F_2 \in \mathcal{F}_2$. In fact, since $(\mathcal F_2,\mathcal L)$ is admissible, the sequence (\ref{equat}) is exact. By Lemma \ref{lem 3} along with the fact that $(\mathcal F_2,\mathcal L)$ is balanced, the sequence \eqref{equat} is also $\Hom_{\mathcal C}(-,\mathcal L)$-acyclic. But then, using now that $(\mathcal F_1,\mathcal L)$ is balanced, \eqref{equat} is also $\Hom_{\mathcal C}(\mathcal F_1,-) $-acyclic. This implies that \eqref{equat} splits, since $F_1 \in \mathcal{F}_1$. Hence, $F_1\in \mathsf{Smd}(\mathcal F_2)$, which completes the proof.
 \epf

%%%%%%%%%%%%%%%%%%%%%%%%%%%%%%%%%%%%%%%%%%%%%%%%%%
%%%%%%%%%%%%%%%%%%%%%%%%%%%%%%%%%%%%%%%%%%%%%%%%%%
%%%%%%%%%%%%%%%%%%COTORSION%TRIPLETS%%%%%%%%%%%%%%%%%%%
%%%%%%%%%%%%%%%%%%%%%%%%%%%%%%%%%%%%%%%%%%%%%%%%%%
%%%%%%%%%%%%%%%%%%%%%%%%%%%%%%%%%%%%%%%%%%%%%%%%%%

\section{Relation between balanced pairs and cotorsion triplets}

It is not in general an easy task to check whether or not a pair of classes $(\mathcal F,\mathcal L)$ form a balanced pair in an abelian category. A common source to provide with such pairs is by means of \emph{cotorsion triplets}. This section is thus devoted to define such triplets and to explore their relation with balanced pairs. In summary, every complete and hereditary cotorsion triplet gives rise to a balanced pair. Cotorsion triplets were introduced by A. Beligiannis and I. Reiten in \cite[Section~3 of Chapter~VI.]{BR}, where they study necessary and sufficient conditions for the existence of such triplets. The concept is also studied by Enochs and Jenda in \cite[Section~4.2]{EJ2} in the context of chain complexes of modules over an associative ring with identity.

\begin{df}\label{df:triplets}
Three classes $\mathcal{F}$, $\mathcal{G}$ and $\mathcal{L}$ of objects in $\mathcal{C}$ form a \textbf{cotorsion triplet} $(\mathcal{F,G,L})$ if $(\mathcal{F,G})$ and $(\mathcal{G,L})$ are cotorsion pairs in $\mathcal{C}$. Moreover, a cotorsion triplet $(\mathcal{F,G,L})$ in $\mathcal{C}$ is:
\begin{enumerate}
\item \textbf{Complete} if $(\mathcal{F,G})$ and $(\mathcal{G,L})$ are complete cotorsion pairs.

\item \textbf{Hereditary} if $(\mathcal{F,G})$ and $(\mathcal{G,L})$ are hereditary cotorsion pairs.
\end{enumerate}
\end{df}

%%%%%%%%%%%%%%%%%%%%%%%%%%%%%%%%%%%%%%%%%%%%
%%%%%%%%%%%%%%%%%%%%%%%%%%%%%%%%%%%%%%%%%%%%

\subsection*{From cotorsion triplets to balanced pairs}

The relation between cotorsion triplets and balanced pairs is summarized in the next proposition. It was originally outlined by Enochs, Jenda and Torrecillas in \cite[Theorem 4.1]{EJTX}, but the precise formulation we state below is due to Chen \cite[Proposition 2.6]{Chen}.

\begin{prop}\label{prop.relation.balance.triplets}
If $(\mathcal{F,G,L})$ is a complete hereditary cotorsion triplet in $\mathcal{C}$, then $(\mathcal{F,L})$ is an admissible balanced pair in $\mathcal{C}$.
\end{prop}

\begin{remark}
Chen's original statement and proof in \cite[Proposition 2.6]{Chen} requires that $\mathcal{C}$ has enough projectives and injectives. However, these hypotheses are actually not necessary. This fact has to do with an interesting characterization of abelian categories with enough projectives and injectives in terms of complete hereditary cotorsion triplets, presented in Theorem~\ref{theorem.cotorsion.triplets} below.

In particular, this result shows that it is hopeless to look for complete hereditary cotorsion triplets in Grothendieck categories without enough projectives, such as some interesting categories studied in Algebraic Geometry. For example, if $T$ is a non-trivial topological space and $\mathcal{O}$ is a sheaf of commutative rings with 1 on $T$, then $\textbf{Sh}(\mathcal{O})$, the category of sheaves of $\mathcal{O}$-modules, does not have enough projective $\mathcal{O}$-modules. This is also the case of the category $\Qcoh(X)$ of quasi-coherent sheaves on a non-affine scheme $X$, considered in Section~\ref{sec:flat_balance}. Thus, it will follow that neither $\textbf{Sh}(\mathcal{O})$ nor $\Qcoh(X)$ have complete and hereditary cotorsion triplets.
\end{remark}

\begin{thm}\label{theorem.cotorsion.triplets}
The following conditions are equivalent.
\begin{itemize}
\item[(a)] $\mathcal{C}$ has enough projectives and injectives.

\item[(b)] There exists a complete hereditary cotorsion triplet $(\mathcal{F,G,L})$ in $\mathcal{C}$.
\end{itemize}
\end{thm}

\begin{proof}
For the implication (a) $\Rightarrow$ (b) it suffices to consider the complete hereditary cotorsion triplet $(\mathsf{Proj}(\mathcal{C}),\mathcal{C},\mathsf{Inj}(\mathcal{C}))$.

Let us now prove (b) $\Rightarrow$ (a). So suppose we are given a complete hereditary cotorsion triplet $(\mathcal{F,G,L})$ in $\mathcal{C}$. For any object $C \in \mathcal{C}$, we have a short exact sequence
\[
0 \to L \to G \to C \to 0
\]
in $\mathcal{C}$ with $G \in \mathcal{G}$ and $L \in \mathcal{L}$, since $(\mathcal{G,L})$ is a complete cotorsion pair. Now using the completeness of $(\mathcal{F,G})$, we have a short exact sequence
\[
0 \to G' \to F \to G \to 0
\]
with $F \in \mathcal{F}$ and $G' \in \mathcal{G}$. Note that $F$ actually belongs to $\mathcal{F} \cap \mathcal{G}$ since $\mathcal{G}$ is closed under extensions. Now taking the pullback of $L \rightarrow G \leftarrow F$, we obtain two short exact sequences of the form:
\begin{align}
0 & \to G' \to K \to L \to 0 \label{eqn:sec1} \\
0 & \to K \to F \to C \to 0 \label{eqn:sec2}
\end{align}
Note  that $G', L \in (\mathcal{F} \cap \mathcal{G})^{\perp}$ in \eqref{eqn:sec1}, and so $K \in (\mathcal{F} \cap \mathcal{G})^{\perp}$. The proof will conclude after we show that $\mathcal{F} \cap \mathcal{G} = \mathsf{Proj}(\mathcal{C})$. The containment $(\supseteq)$ is clear. Now let $W \in \mathcal{F} \cap \mathcal{G}$. From \eqref{eqn:sec2} we have the long homology exact sequence
\[
\cdots \to \Ext^i_{\mathcal{C}}(W,F) \to \Ext^i_{\mathcal{C}}(W,C) \to \Ext^{i+1}_{\mathcal{C}}(W,K) \to \cdots.
\]
On the one hand, $\Ext^i_{\mathcal{C}}(W,F) = 0$ for every $i > 0$ since $W \in \mathcal{F}$ and $F \in \mathcal{G}$, and $(\mathcal{F,G})$ is a hereditary cotorsion pair. On the other hand, $\Ext^{i+1}_{\mathcal{C}}(W,K) = 0$ for every $i > 0$ since $W \in \mathcal{F} \cap \mathcal{F}$ and $K \in (\mathcal{F} \cap \mathcal{F})^\perp$. It follows that $\Ext^i_{\mathcal{C}}(W,C) = 0$ for every positive integer $i > 0$. Since the object $C \in \mathcal{C}$ is arbitrary, we have that $W \in \mathsf{Proj}(\mathcal{C})$.

A dual argument shows that $\mathcal{C}$ has also enough injectives.
\end{proof}

From now on, unless otherwise specified, $R$ will be an associative ring with identity, and all modules are left $R$-modules.

\begin{exa}\label{ex:triplets}
We collect from the literature the following examples of complete hereditary cotorsion triplets (and hence of admissible balanced pairs):
\begin{enumerate}
\item Let $\mathcal{C}$ be an abelian category. We already know from the proof of Theorem~\ref{theorem.cotorsion.triplets} that $(\mathsf{Proj}(\mathcal{C}),\mathcal{C},\mathsf{Inj}(\mathcal{C}))$ is a complete cotorsion triplet if, and only if, $\mathcal{C}$ has enough projectives and injectives. If any of these two conditions holds, we have the well known balanced pair $(\mathsf{Proj}(\mathcal{C}),\mathsf{Inj}(\mathcal{C}))$. Not all of the complete hereditary cotorsion triplets in $\mathcal{C}$ have to be of the form $(\mathsf{Proj}(\mathcal{C}),\mathcal{C},\mathsf{Inj}(\mathcal{C}))$, as shown in the rest of the examples.

\item \label{ex:quasiFrobenius} Consider the category $\mathsf{Mod}(R)$ of modules. In this case, let us set $\mathsf{Proj}(\mathsf{Mod}(R)) = \mathsf{Proj}(R)$ and $\mathsf{Inj}(\mathsf{Mod}(R)) = \mathsf{Inj}(R)$, for simplicity. Recall that a ring $R$ is quasi-Frobenius if $\mathsf{Proj}(R) = \mathsf{Inj}(R)$. We can note that $R$ is quasi-Frobenius if, and only if, the triplet $(\mathsf{Mod}(R),\mathsf{Proj}(R),\mathsf{Mod}(R))$ is a complete cotorsion triplet.

\item Beligiannis and Reiten \cite[Section~3 of Chapter~VI.]{BR}: Let $\Lambda$ be an Artin algebra and $\mathsf{mod}(\Lambda)$ denote the abelian category of finitely generated left $\Lambda$-modules. Let $\mathsf{add}(\Lambda)$ denote the class of objects in $\mathsf{mod}(\Lambda)$ that are direct summands of finite direct sums of copies of $\Lambda$. The class $\mathsf{CM}(\Lambda)$ of \emph{maximal Cohen-Macaulay modules over $\Lambda$} is defined as those $M \in \mathsf{mod}(\Lambda)$ such that there exists an exact sequence
\[
0 \to M \to W^0 \xrightarrow{f^0} W^1 \xrightarrow{f^1} W^2 \to \cdots
\]
with $W^k \in \mathsf{add}(\Lambda)$ and ${\rm Ker}(f^k) \in {}^\perp(\mathsf{add}(\Lambda))$ for every $k \geq 0$. The class $\mathsf{CoCM}(\Lambda)$ is defined dually. On the other hand, let $\mathsf{proj}_{\infty}(\Lambda)$ (respectively $\mathsf{inj}_{\infty}(\Lambda)$) denote the class of finitely generated $\Lambda$-modules with finite projective (respectively injective) dimension. If $\Lambda$ is Gorenstein, then $({\rm CM}(\Lambda),\mathsf{proj}_{\infty}(\Lambda),{\rm CoCM}(D(\Lambda)))$ is a complete hereditary cotorsion triplet in $\mathsf{mod}(\Lambda)$, where $D(\Lambda)$ is the minimal injective cogenerator of $\mathsf{mod}(\Lambda)$. In this case, one has $\mathsf{proj}_{\infty}(\Lambda) = \mathsf{inj}_{\infty}(\Lambda)$.

\item Enochs and Jenda \cite[Proposition 4.4.5]{EJ2}: Let $\mathsf{Ch}(R)$ denote the category of chain complexes of modules. Recall from \cite[Definition 4.2.2]{EJ2} that a chain complex $P = (P_m,\partial^P_m)_{m \in \mathbb{Z}}$ is \emph{perfect} if $P_m = 0$ except for a finite number of integers $m \in \mathbb{Z}$ and if each $P_m$ is a finitely generated projective module. If $\mathcal{S}$ is a set of perfect complexes and $\mathcal{U}$ is the set of all complexes $\Sigma^k(P)$ where $P \in \mathcal{S}$ and $k \in \mathbb{Z}$, then there exists a unique complete hereditary cotorsion triplet $(\mathcal{Y,X,Z})$ in $\mathsf{Ch}(R)$ where $\mathcal{X} = \mathcal{U}^\perp$. Here, $\Sigma^k(P)$ denotes the \emph{$k$-th suspension} of $P$, that is, $\Sigma^k(P)_m := P_{m-k}$ for every integer $m \in \mathbb{Z}$, with boundaries given by $(-1)^k \partial^{\bm{P}}_{m-k}$.
\end{enumerate}
The following examples seem to show that every type of relative homological algebra has its own associated balanced pair.
\begin{enumerate}
\setcounter{enumi}{4}
\item \cite[Section~4.3 of Chapter IV.]{EJ2}: Let $\mathcal{E}$ denote the class of exact chain complexes in $\mathsf{Ch}(R)$. Then, $({}^{\perp_1}\mathcal{E},\mathcal{E},\mathcal{E}^{\perp_1})$ is a complete hereditary cotorsion triplet in $\mathsf{Ch}(R)$, known as the \emph{Dold triplet}. Here, ${}^{\perp_1}\mathcal{E}$ coincides with the class $\mathsf{dg}(\mathsf{Proj}(R))$ of DG-projective complexes in $\mathsf{Ch}(R)$, defined as those complexes $P$ in $\mathsf{Ch}(R)$ such that $P_m$ is a projective module for every integer $m \in \mathbb{Z}$, and every chain map $P \to E$ is homotopic to zero whenever $E \in \mathcal{E}$. Dually, $\mathcal{E}^{\perp_1}$ coincides with the class $\mathsf{dg}(\mathsf{Inj}(R))$ of DG-injective complexes. Here, we have the balanced pair $(\mathsf{dg}(\mathsf{Proj}(R)),\mathsf{dg}(\mathsf{Inj}(R)))$.

\item \label{ex:GorensteinTriplet} Hovey \cite[Section~8]{Hovey}: Let $\mathsf{GProj}(R)$ and $\mathsf{GInj}(R)$ denote the classes of Gorenstein projective and Gorenstein injective modules. Let $\mathsf{Proj}_{\infty}(R)$ (respectively $\mathsf{Inj}_{\infty}(R)$) denote the class of modules with finite projective (respectively injective) dimension. If $R$ is an Iwanaga-Gorenstein ring, then $(\mathsf{GProj}(R), \mathsf{Proj}_{\infty}(R), \mathsf{GInj}(R))$ is a complete hereditary cotorsion triplet in $\mathsf{Mod}(R)$, where $\mathsf{Proj}_{\infty}(R) = \mathsf{Inj}_{\infty}(R)$ by \cite[Proposition~9.1.7]{EJ}. Here, we have the balanced pair $(\mathsf{GProj}(R),\mathsf{GInj}(R))$ comprising several properties in Gorenstein homological algebra.

\item Gillespie \cite{Gi}: Similar to \eqref{ex:GorensteinTriplet} above, let $\mathsf{DProj}(R)$ and $\mathsf{DInj}(R)$ denote the classes of Ding-projective and Ding-injective modules, respectively. Let $\mathsf{Flat}_{\infty}(R)$ (respectively ${\rm FP}\mbox{-}\mathsf{Inj}_{\infty}(R)$) denote the class of modules with finite flat (respectively FP-injective) dimension. If $R$ is a Ding-Chen ring, then $(\mathsf{DProj}(R),\mathsf{Flat}_{\infty}(R),\mathsf{DInj}(R))$ is a complete hereditary cotorsion triplet in $\mathsf{Mod}(R)$, where $\mathsf{Flat}_{\infty}(R) = {\rm FP}\mbox{-}\mathsf{Inj}_{\infty}(R)$ by \cite[Proposition 3.16]{DC}. In this case, we have the balanced pair $(\mathsf{DProj}(R),\mathsf{DInj}(R))$ for Ding-Chen homological algebra.
\end{enumerate}
\end{exa}

%%%%%%%%%%%%%%%%%%%%%%%%%%%%%%%%%%%%%%%%%%%%%
%%%%%%%%%%%%%%%%%%%%%%%%%%%%%%%%%%%%%%%%%%%%%

\subsection*{From balanced pairs to cotorsion triplets}

In \cite[Open Problems]{EJTX} is asked under what conditions a converse of Proposition \ref{prop.relation.balance.triplets} holds. Namely, giving a special precovering class $\mathcal{F}$ and a special preenveloping class $\mathcal{L}$ in $\mathcal{C}$ such that the pair $(\mathcal{F,L})$ is balanced, under what conditions is it true that we have a complete cotorsion triplet $(\mathcal{F,G,L})$?. In the next proposition, we give sufficient conditions on such $\mathcal F$ and $\mathcal L$ to ensure that they are the extremes of a complete hereditary cotorsion triplet.

\begin{prop}\label{prop:from_balance_to_triplets}
Let $\mathcal{C}$ be an abelian category with enough projectives and injectives. Let $\mathcal F$ and $\mathcal L$ be two classes of objects in $\mathcal C$ closed under direct summands such that:
\begin{enumerate}
\item The class $\mathcal F$ is resolving and special precovering, and the class $\mathcal L$ is coresolving and special preenveloping.

\item $\mathcal F\cap \mathcal F^\perp\subseteq {^{\perp}\mathcal L}$ and $^{\perp}\mathcal L\cap \mathcal L\subseteq \mathcal F^\perp$.

\item The pair $(\mathcal F,\mathcal L)$ is balanced.
\end{enumerate}
Then, there is a complete hereditary complete cotorsion triplet $(\mathcal{F,G,L})$ in $\mathcal{C}$. In this case, we have $\mathcal{F} \cap \mathcal{F}^\perp = \mathsf{Proj}(\mathcal{C})$ and ${}^\perp\mathcal{L} \cap \mathcal{L} = \mathsf{Inj}(\mathcal{C})$.
\end{prop}

\pf
Let us call $\mathcal{H}=\mathcal{F}^{\perp}$ and $\mathcal G={^{\perp}\mathcal L}$. With the hypothesis on $\mathcal F$ and $\mathcal L$ we get from Lemma \ref{lema.complete.hered.sprecovering} that $(\mathcal F,\mathcal H)$ and $(\mathcal G,\mathcal L)$ are complete hereditary cotorsion pairs in $\mathcal{C}$. Let us see that $\mathcal H=\mathcal G$. For any $H \in  \mathcal{H}$, we have a $\Hom_{\mathcal{C}}(\mathcal{F},-)$ exact sequence
\[
0\to H_0\to F\to H\to 0,
\]
with $F \in \mathcal{F}$ and $H_0\in\mathcal{H}$. It follows that $F\in \mathcal{F}\cap\mathcal{H}\subseteq \mathcal{G}$ by hypothesis. By Lemma \ref{lem 3}, the above sequence is also $\Hom_{\mathcal{C}}(-,\mathcal{L})$ exact, so we get $H\in \mathcal G $. So $\mathcal{H}\subseteq\mathcal{G}$. Dually, we also have that $\mathcal{G}\subseteq\mathcal{H}$.
\epf

\begin{remark}\label{rem:quasiFrobenius}
As mentioned in the introduction, one cannot expect to obtain a complete hereditary cotorsion triplet from any balanced pair. After checking the statement of Proposition~\ref{prop:from_balance_to_triplets}, it seems difficult to obtain such triplets from a balanced pair $(\mathcal{F,L})$ without assuming condition (2). For example, for any ring $R$ we have the trivial balanced pair $(\mathsf{Mod}(R),\mathsf{Mod}(R))$ by setting $\mathcal{F} = \mathcal{L} = \mathsf{Mod}(R)$. However, we know from Example~\ref{ex:triplets}~\eqref{ex:quasiFrobenius} that the triplet $(\mathsf{Mod}(R),\mathcal G,\mathsf{Mod}(R))$ is complete if, and only if, $R$ is quasi-Frobenius. Note that in this case, we have $\mathcal{F} \cap \mathcal{F}^\perp = \mathsf{Inj}(R)$ and ${}^\perp\mathcal{L} \cap \mathcal{L} = \mathsf{Proj}(R)$, and thus condition (2) in Proposition~\ref{prop:from_balance_to_triplets} holds if, and only if, $R$ is quasi-Frobenius.
\end{remark}

As an immediate consequence of Propositions~\ref{prop.relation.balance.triplets} and \ref{prop:from_balance_to_triplets} we get the following.

\begin{cor}\label{cor.b.pairs.cot.triplets}
Let $\mathcal C$ be an abelian category with enough projectives and injectives. If $(\mathcal F,\mathcal H)$ and $(\mathcal G,\mathcal L)$ are complete hereditary cotorsion pairs in $\mathcal{C}$ with $\mathcal F\cap\mathcal H\subseteq \mathcal G$ and $\mathcal G\cap \mathcal L\subseteq \mathcal H$, then $\mathcal H=\mathcal G$ if and only if $(\mathcal F,\mathcal L)$ is a balanced pair in $\mathcal{C}$.
\end{cor}

%%%%%%%%%%%%%%%%%%%%%%%%%%%%%%%%%%%%%%%%%%
%%%%%%%%%%%%%%%%%%%%%%%%%%%%%%%%%%%%%%%%%%

\subsection*{Virtually Gorensteins rings, balanced pairs and cotorsion triplets}

We close this section presenting a first application of the relation between balanced pairs and cotorsion triplets described in Propositions~\ref{prop.relation.balance.triplets} and \ref{prop:from_balance_to_triplets}, in the context of virtually Gorenstein rings (a notion originally due to Beligiannis and Reiten in \cite{BR} for Artin algebras). More applications will be given later on for the categories of quasi-coherent sheaves and $\mathcal{C}$-valued representations of quivers. These two settings will be studied in more detail in Sections \ref{sec:flat_balance} and \ref{sec:quiver}, respectively.

The balanced pair $(\mathsf{GProj}(R),\mathsf{GInj}(R))$ from Example~\ref{ex:triplets} \eqref{ex:GorensteinTriplet} can be obtained under different assumptions on $R$. As a matter of fact, the existence of $(\mathsf{GProj}(R),\mathsf{GInj}(R))$ as a balanced pair in $\mathsf{Mod}(R)$ is a necessary and sufficient condition for certain rings $R$ to be virtually Gorenstein. Recall that a (non-necessarily commutative) ring $R$ is called \emph{virtually Gorenstein} provided that $(\mathsf{GProj}(R))^{\perp}= {^{\perp}} (\mathsf{GInj}(R))$.

In the case where $R$ is a Noetherian ring of finite Krull dimension, it is proved by Zareh-Khoshchehreh, Asgharzadeh and Divaani-Aazar in \cite[Theorem 3.10]{ZAD} that $R$ is virtually Gorenstein if, and only if, $(\mathsf{GProj}(R),\mathsf{GInj}(R))$ is a balanced pair in $\mathsf{Mod}(R)$. This is an important recent result for which we will present two extensions in Corollaries~\ref{cor:ZAD1} and \ref{cor:ZAD2}. The former adds an extra condition in this equivalence, namely the existence of a cotorsion triplet $(\mathsf{GProj}(R),\mathcal{G},\mathsf{GInj}(R))$ in $\mathsf{Mod}(R)$. For the latter extension, on the other hand, we will require some concepts and techniques from Representation Theory of Quivers,  covered in Section~\ref{sec:quiver}.

\begin{cor}\label{cor:ZAD1}
Let $R$ be a commutative Noetherian ring with finite Krull dimension. Then, the following conditions are equivalent.
\begin{itemize}
\item[(a)] $R$ is a virtually Gorenstein ring.

\item[(b)] $(\mathsf{GProj}(R),\mathsf{GInj}(R))$ is a balanced pair in $\mathsf{Mod}(R)$.

\item[(c)] There is a complete hereditary cotorsion triplet $(\mathsf{GProj}(R),\mathcal{G},\mathsf{GInj}(R))$ in $\mathsf{Mod}(R)$.
\end{itemize}
\end{cor}

\begin{proof}
The equivalence (a) $\Leftrightarrow$ (b) is \cite[Theorem 3.10]{ZAD}, which also holds in the non commutative case. The implication (c) $\Rightarrow$ (b) is an immediate consequence of Proposition~\ref{prop.relation.balance.triplets}. So the proof will conclude after showing (b) $\Rightarrow$ (c).

Suppose that the classes $\mathsf{GProj}(R)$ and $\mathsf{GInj}(R)$ form a balanced pair $(\mathsf{GProj}(R),\mathsf{GInj}(R))$. Firstly, it is well known for any arbitrary ring $R$ that the classes $\mathsf{GProj}(R)$ and $\mathsf{GInj}(R)$ are resolving and coresolving, respectively, and that $\mathsf{GProj}(R) \cap (\mathsf{GProj}(R))^\perp = \mathsf{Proj}(R) \subseteq {}^\perp(\mathsf{GInj}(R))$ and ${}^\perp(\mathsf{GInj}(R)) \cap \mathsf{GInj}(R) = \mathsf{Inj}(R) \subseteq (\mathsf{GProj}(R))^\perp$. Moreover, since $R$ is Noetherian we have by Krause \cite[Theorem 7.12]{Krause} that $\mathsf{GInj}(R)$ is special preenveloping. On the other hand, since also $R$ is commutative with finite Krull dimension, we have that $\mathsf{GProj}(R)$ is special precovering (see e.g. \cite[Proposition 6]{EIO}). Thus, we are under the hypotheses of Proposition~\ref{prop:from_balance_to_triplets}, which says that there must exist a complete hereditary cotorsion triplet $(\mathsf{GProj}(R),\mathcal{G},\mathsf{GInj}(R))$ in $\mathsf{Mod}(R)$.
\end{proof}

%%%%%%%%%%%%%%%%%%%%%%%%%%%%%%%%%%%%%%%%%%%%%%%%%%
%%%%%%%%%%%%%%%%%%%%%%%%%%%%%%%%%%%%%%%%%%%%%%%%%%
%%%%%%%%%%%%%%BALANCE%WITH%FLAT%OBJECTS%%%%%%%%%%%%%%%%%
%%%%%%%%%%%%%%%%%%%%%%%%%%%%%%%%%%%%%%%%%%%%%%%%%%
%%%%%%%%%%%%%%%%%%%%%%%%%%%%%%%%%%%%%%%%%%%%%%%%%%

\section{Balance with flat objects}\label{sec:flat_balance}

In this section, we first give a different proof to that of Enochs in \cite[Theorem 4.1]{Enochs} about the lack of balance with respect to the class of flat modules, in case the ring $R$ is left Noetherian and non-perfect.

%%%%%%%%%%%%%%%%%%%%%%%%%%%%%%%%%%%%%%%%%%%%%%%%%%
%%%%%%%%%%%%%%%%%%%%%%%%%%%%%%%%%%%%%%%%%%%%%%%%%%

\subsection*{Balance and closure under direct sums and products}

We start with the following consequence of balance in abelian categories. We recall that an abelian category satisfies AB4 if it is cocomplete and any direct sum of monomorphisms is a monomorphism. The axiom AB4* of an abelian category is dual.

\begin{lem}\label{lemma.noflatbalance}
Let $\mathcal F$ and $\mathcal L$ be two classes of objects in $\mathcal C$ such that $(\mathcal{F,L})$ is a balanced pair. Then, the following statements hold:
\begin{enumerate}
\item If $\mathcal C$ satisfies AB4, has enough injectives and any direct sum of injective objects belongs to $\mathcal F^{\perp_1}$, then $\mathcal F^{\perp_1}$ is closed under direct sums.

\item If $\mathcal C$ satisfies AB4*, has enough projectives and any direct product of projective objects belongs to $^{\perp_1}\mathcal L$, then $^{\perp_1}\mathcal L$ is closed under direct products.
\end{enumerate}
\end{lem}

\pf Let $\{ C_i\}$ be a family of objects in $\mathcal{F}^{\perp_1}$ and
\[
0\to C_i\to E_i\to D_i\to 0
\]
be a family of exact sequences with each $E_i$ injective. Since each $C_i \in \mathcal{F}^{\perp_1}$, each of these sequences is ${\Hom}_{\mathcal C}(\mathcal F,-)$-exact. Hence by Lemma \ref{lem 3}, they will be ${\Hom}_{\mathcal C}(-,\mathcal L)$-exact. So, for each $i$ and each $L\in \mathcal L$, we have the exact sequence of abelian groups
\[
0\to {\Hom}_{\mathcal C}(D_i,L)\to {\Hom}_{\mathcal C}(E_i,L)\to {\Hom}_{\mathcal C}(C_i,L)\to 0.
\]
We can take the direct product of the previous family of short exact sequences to get the exact sequence
\[
0 \to \prod_i {\Hom}_{\mathcal C}(D_i,L)\to \prod_i{\Hom}_{\mathcal C}(E_i,L)\to \prod_i{\Hom}_{\mathcal C}(C_i,L)\to 0.
\]
Now, we have the following commutative diagram

\[
\xymatrix{
0 \ar[r] & \displaystyle\operatorname*{\prod}_i {\Hom}_{\mathcal{C}}(D_i,L) \ar[d]^{\simeq}\ar[r] & \displaystyle\operatorname*{\prod}_i {\Hom}_{\mathcal{C}}(E_i,L) \ar[d]^{\simeq} \ar[r] & \displaystyle\operatorname*{\prod}_i {\Hom}_{\mathcal{C}}(C_i,L) \ar[d]^{\simeq} \ar[r] & 0 \\
0 \ar[r] & {\Hom}_{\mathcal{C}}\Big(\displaystyle\operatorname*{\bigoplus}_i D_i, L\Big) \ar[r] & {\Hom}_{\mathcal{C}}\Big(\displaystyle\operatorname*{\bigoplus}_i E_i,L\Big) \ar[r] & {\Hom}_{\mathcal{C}}\Big(\displaystyle\operatorname*{\bigoplus}_i C_i,L\Big) \ar[r] & 0
}
\]
where the columns are natural isomorphisms. The bottom row tells us that the exact sequence
\[
0 \to \bigoplus_i C_i\to \bigoplus_i E_i \to \bigoplus_i D_i \to 0
\]
is ${\Hom}_{\mathcal{C}}(-,\mathcal L)$-exact. Since $(\mathcal F, \mathcal L)$ is balanced, by applying Lemma \ref{lem 3} again, it follows that the sequence is ${\Hom}_{\mathcal{C}}(\mathcal F,-)$-exact. Since $\oplus_i E_i\in \mathcal F^{\perp_1}$ by hypothesis, it follows from the usual long exact sequence of cohomology that $\Ext^1_{\mathcal C}(F,\oplus_i C_i) = 0$ for each $F\in \mathcal{F}$, that is, $\oplus_i C_i \in \mathcal{F}^{\perp_1}$.

The proof of (2) is dual.
\epf

%%%%%%%%%%%%%%%%%%%%%%%%%%%%%%%%%%%%%%%%%%%%%%%%%%%%
%%%%%%%%%%%%%%%%%%%%%%%%%%%%%%%%%%%%%%%%%%%%%%%%%%%%

\subsection*{Lack of balance with respect to flat modules}

We are now in position to give a short proof of the aforementioned result of \cite[Theorem 4.1]{Enochs}. In what follows, we will denote by $\mathsf{Flat}(R)$ the class of flat left $R$-modules.

\begin{thm}\label{theorem.noflatbalance}
Let $R$ be a left Noetherian ring. The class of flat modules is the left part of a balanced pair if, and only if, the ring $R$ is left perfect.
\end{thm}

\pf
Let us first prove the impication $(\Leftarrow)$. If $R$ is left perfect, then the class of flat modules coincides with the class of projective modules, hence we get the standard balanced pair $(\mathsf{Proj}(R),\mathsf{Inj}(R))$ in $\mathsf{Mod}(R)$.

In order to show the converse implication $(\Rightarrow)$, suppose there is a balanced pair $(\mathsf{Flat}(R),\mathcal L)$ for some class of modules $\mathcal L$. Since $R$ is left Noetherian, any direct sum of injective modules is injective. Therefore, we are in the assumptions of part (1) of Lemma~\ref{lemma.noflatbalance}, that says that the class $(\mathsf{Flat}(R))^{\perp_1}$ of cotorsion modules is closed under direct sums. But then by Guil Asensio and Herzog \cite[Theorem 19]{GH}, the ring $R$ must be left perfect.
\epf

Following the philosophy of \cite[Section~5]{Enochs}, we want to mention other cases for which Theorem~\ref{theorem.noflatbalance} is also valid. First, one can state a chain complex version of  Theorem~\ref{theorem.noflatbalance} by noticing some facts. Firstly, recall that a chain complex is flat if it is exact with flat cycles. Also, projective and injective complexes have similar descriptions. So if $\bm{{\rm Flat}}(R)$ denotes the class of flat complexes, we can note that if $(\bm{{\rm Flat}}(R))^{\perp_1}$ is closed under direct sums, then so will be the class $(\mathsf{Flat}(R))^{\perp_1}$ of cotorsion modules. For it suffices to note that for every cotorsion module $C$, the complex $\underline{C} = \cdots \to 0 \to C \to 0 \to \cdots$ belongs to $(\bm{{\rm Flat}}(R))^{\perp_1}$. This follows applying a well known natural isomorphism appearing in \cite[Lemma 4.2]{GillespieDegree}.

The other context we are interested in is the category of quasi-coherent sheaves on a scheme $X$, presented in the following section.

%%%%%%%%%%%%%%%%%%%%%%%%%%%%%%%%%%%%%%%%%%%%%%%%%%%%
%%%%%%%%%%%%%%%%%%%%%%%%%%%%%%%%%%%%%%%%%%%%%%%%%%%%

\subsection*{Lack of balance with respect to flat quasi-coherent modules on a scheme}

From now until the end of this section all rings are commutative.

Let $\mathfrak{Qcoh}(X)$ denote the category of quasi-coherent sheaves on a scheme $X$. The corresponding version of Theorem~\ref{theorem.noflatbalance} for $\mathfrak{Qcoh}(X)$ is formulated below in Corollary~\ref{cor:balance_scheme}. This result answers the question (6) posted in \cite[Section 6]{Enochs} in the negative.

For a better understanding of Corollary~\ref{cor:balance_scheme}, we need to recall a few well-known facts about $\mathfrak{Qcoh}(X)$. First, a scheme $X$ is called \emph{semi-separated} if it has a \emph{semi-separating} open affine covering $\mathfrak{U}=\{U_i:\ i\in I\}$, that is, for each $i,k\in I$ the intersection $U_i\cap U_k$ is also an open affine. For each $i\in I$, the canonical inclusion $\iota_i\mathcolon U_i\to X$ gives an adjoint pair $(\iota^*_i,\iota_*^i)$, where
\[
\iota^*_i\mathcolon \Qcoh(X) \to \Qcoh(U_i) \mbox{ \ and \ } \iota_*^i\mathcolon \Qcoh(U_i)\to \Qcoh(X)
\]
are the inverse and direct image functors, respectively. In general, the direct image functor $\iota_*^i$ does not preserve quasi-coherence, but it does for semi-separated schemes $X$. So, for each $U_i$, we have an isomorphism
\[
\Hom_{\Qcoh(U_i)}(\iota^*_i\mathscr H,\mathscr{T}) \cong \Hom_{{\Qcoh(X)}}(\mathscr{H},\iota_*^i\mathscr{T}).
\]
Since, for each open affine $U_i$, the categories $\mathsf{Mod}(\R(U_i))$ and $\Qcoh(U_i)$ are equivalent by a well known result of Grothendieck, we can write the previous isomorphism as
\[
\Hom_{\mathcal O_X(U_i)}(\mathscr{H}(U_i),T) \cong \Hom_{{\Qcoh(X)}}(\mathscr{H},\iota_*^i(T)),
\]
for any $\R(U_i)$-module $T$ and any quasi-coherent sheaf $\mathscr H$. We recall that a scheme is \emph{Noetherian} if it is quasi-compact and it possesses an open affine covering $\mathfrak{U}=\{U_1,\ldots,U_n\}$ such that, for each $i=1,\ldots, n$, $\R(U_i)$ is a Noetherian ring.

Let $\mathfrak{Flat}(X)$ denote the class of flat quasi-coherent sheaves over $X$ in the following result.

\begin{cor}\label{cor:balance_scheme}
Let $X$ be a Noetherian and semi-separated scheme, with semi-separating open affine covering $\mathfrak{U} = \{U_1,\ldots,U_n\}$. Assume that $\R(U_i)$ is a Noetherian but not Artinian ring, for some $i \in \{1, \ldots, n\}$. Then, $\mathfrak{Flat}(X)$ is not the left part of a balanced pair in $\mathfrak{Qcoh}(X)$.
\end{cor}

\pf
%Let us prove first that the direct image functor $\iota_i^*: \R(U_i)\to \Qcoh(X)$ preserves flat covers.
Suppose that there is such balanced pair $(\mathfrak{Flat}(X),\mathcal{L})$ in $\Qcoh(X)$, for some class $\mathcal L$.
It is well-known that the category $\Qcoh(X)$ is Grothendieck and so it is cocomplete, satisfies AB4 and has enough injectives. Indeed, since $X$ is Noetherian, the category $\Qcoh(X)$ is locally Noetherian, hence the direct sum of injective objects in $\Qcoh(X)$ is again injective. Therefore, part (1) of Lemma~\ref{lemma.noflatbalance} tells us that the class $(\mathfrak{Flat}(X))^{\perp_1}$ of cotorsion quasi-coherent sheaves, is closed under direct sums. Now let $\{C_k\}$ be a family of cotorsion $\R(U_i)$-modules. By Gillespie \cite[Lemma 6.5]{G} the functor  $\iota^i_* \mathcolon \mathsf{Mod}(\R(U_i)) \to \Qcoh(X)$ preserves cotorsion objects. Hence, the family $\{\iota^i_*(C_k)\}$ is a family of cotorsion quasi-coherent sheaves and thus, by the previous, $\oplus_k \iota^i_*(C_k) \in (\mathfrak{Flat}(X))^{\perp_1}$. We will finish the proof by showing that this implies that $\oplus_k C_k$ is a cotorsion $\R(U_i)$-module. So, by Guil Asensio and Herzog \cite[Theorem 19]{GH}, the ring $\R(U_i)$ must be Artinian. A contradiction.

To show what we claimed, let $F$ be a flat $\R(U_i)$-module. We want to show that the equality $\Ext_{\R(U_i)}^1(F, \oplus_k C_k) = 0$ holds. Firstly, notice that $F = \iota_i^*\iota_*^i(F)$. Then, the isomorphism shown in the proof of \cite[Lemma 6.5]{G} gives
\[
\Ext^1_{\R(U_i)}(F, \oplus_k C_k) \cong \Ext^1_{\Qcoh(X)}(\iota_*^i(F), \iota_*^i(\oplus_kC_k)).
\]
The last Ext functor vanishes, because $\iota_*^i(F)$ is a flat quasi-cohent sheaf (so it belongs to $\mathfrak{Flat}(X)$) and  $\iota^i_*(\oplus_kC_k) \simeq \oplus_k \iota^i_*(C_k) \in (\mathfrak{Flat}(X))^{\perp_1}$, because the functor $\iota^i_*$ commutes with direct sums.
\epf

%%%%%%%%%%%%%%%%%%%%%%%%%%%%%%%%%%%%%%%%%%%%%%%%%%
%%%%%%%%%%%%%%%%%%%%%%%%%%%%%%%%%%%%%%%%%%%%%%%%%%
%%%BALANCE%IN%QUIVER%REPRESENTATIONS%AND%COTORSION%TRIPLETS%%%%%
%%%%%%%%%%%%%%%%%%%%%%%%%%%%%%%%%%%%%%%%%%%%%%%%%%
%%%%%%%%%%%%%%%%%%%%%%%%%%%%%%%%%%%%%%%%%%%%%%%%%%

\section{Balance in quiver representations and cotorsion triplets}\label{sec:quiver}

Throughout this section $\mathcal{C}$ will be an abelian category with enough projectives and injectives that satisfies AB4 and AB4*.

In \cite{HJ} Holm and J\o rgensen have recently proved that, under some conditions on a quiver $Q$, a complete cotorsion pair in $\mathcal C$ induces two complete cotorsion pairs in the abelian category Rep$(Q,\mathcal C)$ of $\mathcal C$-valued representations of $Q$. Taking into account the relation between balanced pairs and cotorsion triplets, it seems natural to expect that balanced pairs in $\mathcal C$ and Rep$(Q,\mathcal C)$ should be also related. Thus we will devote this section to study the relation between balanced pairs in $\mathcal C$ and balanced pairs in Rep$(Q,\mathcal C)$. One of the consequences of our results is that they will lead us to finding new conditions over two complete hereditary cotorsion pairs to form a cotorsion triplet.

%%%%%%%%%%%%%%%%%%%%%%%%%%%%%%%%%%%%%%%%%%%%%%%%%%
%%%%%%%%%%%%%%%%%%%%%%%%%%%%%%%%%%%%%%%%%%%%%%%%%%

\subsection*{Adjoint Functors between $\mathcal{C}$ and $\Rep(Q,\mathcal C)$}

A \emph{quiver} $Q = (Q_0,Q_1, s,t)$ is a directed graph  with vertex set $Q_0$, arrow set $Q_1$ and two maps $s,t$ from $Q_1$ to $Q_0$ which associate to each arrow  $\alpha\in Q_1$ its source $s(\alpha)\in Q_0$ and its target $t(\alpha)\in Q_0$, respectively. The quiver $Q$ is said to be \emph{finite} if $Q_0$ and $Q_1$ are finite.

A \emph{representation} $\mathbb{X} = (\mathbb{X}_i,\mathbb{X}_\alpha)$ of $Q$ over $\mathcal C$, or a \emph{$\mathcal{C}$-valued representation}, is defined by the following data:
\begin{enumerate}
\item To each vertex $i$ in $Q_0$ is associated an object $\mathbb{X}_i \in \mathcal C$.

\item To each arrow $\alpha:i\rightarrow j$ in $Q_1$ is associated a morphism $\mathbb{X}_\alpha: \mathbb{X}_i\rightarrow \mathbb{X}_j$ in $\mathcal{C}$.
\end{enumerate}

A \emph{morphism} $f$ from $\mathbb{X}$ to $\mathbb{Y}$ is a family of morphisms $\{f_i \mathcolon \mathbb{X}_i \rightarrow \mathbb{Y}_i \}_{i \in Q_0}$ such that $\mathbb{Y}_\alpha f_i = f_j\mathbb{X}_\alpha$ for any arrow $\alpha \mathcolon i\rightarrow j\in Q_1$. We will denote by $\Rep(Q,\mathcal C)$ the category of all $\mathcal C$-valued representations of a quiver $Q$.

Define the functor $e_\lambda^i \mathcolon \mathcal C \rightarrow \Rep(Q, \mathcal C)$ as
\[
e_\lambda^i(M)_j := \bigoplus \limits_{Q(i,j)} M 
\]
for every vertex $j \in Q_0$ (see Mitchell \cite[Section 28]{M}) with $Q(i,j)$ the set of paths $p$ in $Q$ such that $s(p) = i$ and $t(p) = j$. Moreover, for an arrow $\alpha \mathcolon j \rightarrow k$, the morphism $e_\lambda^i(M)_\alpha$ is the canonical injection. Dually, the functor $e_i^\rho \colon \mathcal C\rightarrow \Rep(Q, \mathcal C)$ is defined by Enochs and Herzog in \cite{EEGI,EH} as
\[
e_i^\rho(M)_j := \prod \limits_{Q(j,i)} M 
\]
for every vertex $j \in Q_0$.

\begin{lem}\label{lem 2} \cite{EEGI,HJ}
Let $i \in Q_0$ and $(\ )_i \colon \Rep(Q, \mathcal C) \longrightarrow \mathcal C$ be the restriction functor given by  $(\mathbb{X})_i = \mathbb{X}_i$ for any representation $\mathbb{X}$ of $\Rep(Q, \mathcal C)$. Then, the following conditions hold: \\

\begin{enumerate}
\item $(\ )_i $ is a right adjoint of $ e_\lambda^i$ and a left adjoint of $ e^\rho_i$.

\item $\Ext_{\Rep(Q, \mathcal C)}^m(e_\lambda^i(Y), \mathbb{X}) \cong \Ext_{\mathcal C}^m(Y, (\mathbb{X})_i)$ for every  $m \geq 0$. 

\item $\Ext^m_{\Rep(Q, \mathcal C)}(\mathbb{X}, e^\rho_i(Y)) \cong \Ext_{\mathcal C}^m((\mathbb{X})_i, Y)$ for every $m \geq 0$.
\end{enumerate}
\end{lem}

For any representation $(\mathbb{X}_i, \mathbb{X}_{\alpha})$ of $\Rep(Q, \mathcal{C})$, there are induced morphisms
\bc$\varphi_{\mathbb{X}_i}\mathcolon \bigoplus\limits_{t(\alpha)=i}\mathbb{X}_{s(\alpha)}\rightarrow\mathbb{X}_i$ and $\psi_{\mathbb{X}_i}\mathcolon\mathbb{X}_i\rightarrow \prod\limits_{s(\alpha)=i}\mathbb{X}_{t(\alpha)}.$ \ec
We will denote by $\mathbbm{c}_i(\mathbb{X})$ the cokernel of $\varphi_{\mathbb{X}_i}$ and by $\mathbbm{k}_i(\mathbb{X})$ the kernel of $\psi_{\mathbb{X}_i}$. The assignments $\mathbbm{c}_i(-)$  and $\mathbbm{k}_i(-)$ from $\Rep(Q,\mathcal C)$ to $\mathcal C$ are functorial.

\begin{lem}\label{lem 1}\cite[Proposition 5.4]{HJ}
Let $i \in Q_0$ and $\mathbbm{s}_i: \mathcal C\rightarrow \Rep(Q, \mathcal C)$ be the stalk functor given by $\mathbbm{s}_i(Y)_j=\delta_{ij}Y$, where $\delta_{ii}Y = Y$ and $\delta_{ij}Y = 0$ whenever $j \neq i$. Then, we have:
\begin{enumerate}
\item $\mathbbm{s}_i$ is a right adjoint of $\mathbbm{c}_i$ and a left adjoint of $\mathbbm{k}_i$;

\item $\Ext^m_{\Rep(Q, \mathcal C)}(\mathbb{X}, \mathbbm{s}_i(Y)) \cong \Ext_{\mathcal C}^m(\mathbbm{c}_i(\mathbb{X}), Y)$ for every $m \geq 0$, provided that $\varphi_{\mathbb{X}_i}$ is monic.

\item $\Ext_{\Rep(Q, \mathcal C)}^m(\mathbbm{s}_i(Y), \mathbb{X}) \cong \Ext_{\mathcal C}^m(Y, \mathbbm{k}_i(\mathbb{X}))$ for every $m \geq 0$, provided that $\psi_{\mathbb{X}_i}$ is epic.
\end{enumerate}
\end{lem}

\begin{cor}\label{cor 1}
Let $Q$ be a quiver without oriented cycles, and let us fix a vertex $k\in Q_0$. Given a class $\mathcal{L}$ of objects of $\mathcal{C}$, for any $G \in {^{\perp_1}\mathcal{L}}$ there is an exact sequence
\[
0 \to \mathbb{K} \to e_\lambda^k(G) \xrightarrow{\widetilde{id}} \mathbbm{s}_k(G) \to 0
\]
in ${\rm Rep}(Q,\mathcal{C})$ with $\widetilde{id} = \delta_{ki}id_G$. Moreover, for any $\mathbb{X}\in{\Rep}(\mathcal C,Q)$, if $\mathbbm{k}_k(\mathbb{X})\in\mathcal{L}$ and $\psi_{\mathbb{X}_k}$ is epic, then the above sequence  is ${\Hom}_{\Rep(Q,\mathcal C)}(-,\mathbb{X})$ exact.
\end{cor}

%\subsection{The lack of flat balance in ${\Rep}(\mathsf{Mod}(R),Q)$}
%Once we have introduced the functors $e_{\lambda}^i$ and $e_i^{\rho}$ we can obtain as an application of Theorem \ref{theorem.noflatbalance} the lack of flat balance in $%%\Rep(\mathsf{Mod}(R),Q)$.

%%%%%%%%%%%%%%%%%%%%%%%%%%%%%%%%%%%%%%%%%%%%%%%%%%
%%%%%%%%%%%%%%%%%%%%%%%%%%%%%%%%%%%%%%%%%%%%%%%%%%

\subsection*{Induced classes in $\Rep(Q,\mathcal C)$}

Let $\mathcal{L}$ be a class of objects of $\mathcal{C}$. Following \cite{HJ} we denote by
\begin{align*}
\Rep(Q,\mathcal{L}) & :=\{\mathbb{X} \in \Rep(Q,\mathcal C) \ | \ \mathbb{X}_i\in\mathcal{L}\mbox{ for all } i\in Q_0\}, \\
\Phi(\mathcal{L}) & := \{\mathbb{X} \in \Rep(Q,\mathcal{L}) \ | \ \varphi_{\mathbb{X}_i}\mbox{ is monic and } \mathbbm{c}_i(\mathbb{X})\in\mathcal{L} \mbox{ for all } i \in Q_0 \}, \\
\Psi(\mathcal{L}) & := \{\mathbb{X} \in \Rep(Q,\mathcal{L}) \ | \ \psi_{\mathbb{X}_i}\mbox{ is epic and } \mathbbm{k}_i(\mathbb{X})\in\mathcal{L} \mbox{ for all } i \in Q_0 \}.
\end{align*}

For the following result, recall that a quiver $Q$ is \emph{discrete} if there are no arrows between its vertexes. So $Q$ is \emph{non-discrete} if there exist at least two vertexes with at least one arrow between them.

\begin{prop}\label{prop 1}
Let $Q$ be a non-discrete quiver without oriented cycles. With the notation above, assume that $(\Phi(\mathcal{F}),\Psi(\mathcal{L}))$ is a balanced pair in $\Rep(Q,\mathcal C)$ for certain classes $\mathcal{F}$ and $\mathcal{L}$ in $\mathcal C$. Then, the following statements holds:
\begin{enumerate}
\item $(\mathcal{F},\mathcal{L})$ is a balanced pair in $\mathcal{C}$.

\item If $\mathcal{F}$ is resolving, then ${^{\perp_1}\mathcal{L}}\subseteq\mathcal{F}^{\perp_1}$.

\item If $\mathcal{L}$ is coresolving, then $\mathcal{F}^{\perp_1}\subseteq{^{\perp_1}\mathcal{L}}$.
\end{enumerate}
\end{prop}

\pf
Let us prove (1) and (2). Part (3) is dual to (2).
\begin{enumerate}
\item  For any object $M \in \mathcal C$, there is a $\Phi(\mathcal{F})$-precover $\sigma\mathcolon \mathbb{F}\rightarrow\mathbbm{s}_i(M)$. Then, we claim that $\widetilde{\sigma}_i\mathcolon\mathbb{F}_i\rightarrow\mathbbm{s}_i(M)_i=M$ is an $\mathcal{F}$-precover of $M$, where $\widetilde{\sigma}_i$ is induced by $\sigma$. In fact, for any $F\in\mathcal{F}$, one can note that the representation $e_\lambda^i(F) $ belongs to $\Phi(\mathcal{F})$. Then, we have an epimorphism
\[
{\Hom}_{\Rep(Q,\mathcal C)}(e_\lambda^i(F), \mathbb{F})\rightarrow{\Hom}_{\Rep(Q,\mathcal C)}(e_\lambda^i(F), \mathbbm{s}_i(M))
\]
which implies by Lemma~\ref{lem 2} an epimorphism ${\Hom}_{\mathcal C}(F, \mathbb{F}_i)\rightarrow{\Hom}_{\mathcal C}(F, M)$, as desired.

Since $(\Phi(\mathcal{F}),\Psi(\mathcal{L}))$ is a balanced pair in ${\rm Rep}(Q,\mathcal{C})$ and $e^\rho_i(L)\in\Psi(\mathcal{L})$ for any $L\in \mathcal{L}$, we have by Lemma \ref{lem 3} an exact sequence
\begin{align*}
\hspace{1.5cm} 0 \rightarrow \Hom_{\Rep(Q,\mathcal C)}(\mathbbm{s}_i(M),e^\rho_i(L)) \rightarrow \Hom_{\Rep(Q,\mathcal C)}(\mathbb{F},e^\rho_i(L)) \rightarrow \Hom_{\Rep(Q,\mathcal C)}(\mathbb{K},e^\rho_i(L)) \rightarrow 0
\end{align*}
with $\mathbb{K}=\ker(\sigma)$. Now by part (1) of Lemma \ref{lem 2}, we have an exact sequence
\[
0\rightarrow{\Hom}_{\mathcal C}(M,L)\rightarrow{\Hom}_{\mathcal C}(\mathbb{F}_i,L)\rightarrow{\Hom}_{\mathcal C}(\mathbb{K}_i,L)\rightarrow0.
\]
Thus the left exact sequence
\[
0 \rightarrow \mathbb{K}_i \rightarrow \mathbb{F}_i \rightarrow M \rightarrow 0
\]
is ${\Hom}_{\mathcal{C}}(-,\mathcal{L})$ and ${\Hom}_{\mathcal{C}}(\mathcal{F},-)$ exact. Similarly, we have that $\mathcal{L}$ is preenveloping and that there is a right exact sequence $0\rightarrow M\rightarrow L\rightarrow C\rightarrow0$ in $\mathcal{C}$, which is ${\Hom}_{\mathcal C}(\mathcal{F},-)$-acyclic and ${\Hom}_{\mathcal C}(-,\mathcal{L})$-acyclic. Therefore, by Lemma \ref{lem 3} the pair $(\mathcal{F,L})$ is balanced.

\item Before proving the statement, we need to make some observations.
\begin{itemize}
\item Since $Q$ is non-discrete, we can fix a non-sink vertex $k \in Q_0$. This means that there exists at least an arrow $k \to i$ in $Q$.

\item Let  $F\in\mathcal{F}$ and  $\sigma\mathcolon P\rightarrow F$ be an epimorphism with $P$ projective. Then, we have an induced epimorphism
\[
\widetilde{\sigma}\mathcolon\mathbb{P} = e_\lambda^{k}(P) \rightarrow \mathbbm{s}_{k}(F) = \mathbb{F}
\]
in $\Rep(Q,\mathcal C)$ with $\widetilde{\sigma}_i = \delta_{k i} \sigma$ for any vertex $i \in Q_0$. Let $\mathbb{K}=\ker(\widetilde{\sigma})$ and let us show $\mathbb{K}\in\Phi(\mathcal{F})$.

For each vertex $i \in Q_0$, we have the following exact commutative diagram in $\mathcal{C}$:
\[
\begin{tikzpicture}
\matrix (m) [matrix of math nodes, row sep=3em, column sep=3em, text height=1.25ex, text depth=0.5ex]
{
0 & \bigoplus \limits_{t(\alpha) = i} \mathbb{K}_{s(\alpha)} & \bigoplus\limits_{t(\alpha) = i} \mathbb{P}_{s(\alpha)} & \bigoplus \limits_{t(\alpha) = i} \mathbb{F}_{s(\alpha)} & 0 \\
0 & \mathbb{K}_i & \mathbb{P}_i & \mathbb{F}_i & 0 \\
};
\path[->]
(m-1-2) edge node[right] {\footnotesize$\varphi_{{\mathbb{K}}_i}$} (m-2-2)
(m-1-3) edge node[right] {\footnotesize$\varphi_{{\mathbb{P}}_i}$} (m-2-3)
(m-1-4) edge node[right] {\footnotesize$\varphi_{{\mathbb{F}}_i}$} (m-2-4)
(m-1-1) edge (m-1-2) (m-1-2) edge (m-1-3) (m-1-3) edge (m-1-4) (m-1-4) edge (m-1-5)
(m-2-1) edge (m-2-2) (m-2-2) edge (m-2-3) (m-2-3) edge (m-2-4) (m-2-4) edge (m-2-5)
;
\end{tikzpicture}
\]
Since $\mathcal F$ contains the projective objects, we follow that $\mathbb{P}=e_\lambda^k(P) $ belongs to $\Phi(\mathcal{F})$. So, in particular, the morphism $\varphi_{\mathbb{P}_i}$ is monic for any vertex $i \in Q_0$. It follows that $\varphi_{\mathbb{K}_i}$ is monic since $\varphi_{\mathbb{P}_i}$ is monic. By the snake lemma, we get the exact sequence
\[
\xymatrix{
0 \ar[r] & \ker(\varphi_{\mathbb{F}_i}) \ar[r] &  \mathbbm{c}_i(\mathbb{K}) \ar[r] & \mathbbm{c}_i(\mathbb{P}) \ar[r] & \mathbbm{c}_i(\mathbb{F}) \ar[r] & 0.
}
\]
Note that $\ker(\varphi_{\mathbb{F}_i}),\mathbbm{c}_i(\mathbb{P}), \mathbbm{c}_i(\mathbb{F})\in\mathcal{F}$ and $\mathcal{F}$ is resolving. It follows that $\mathbbm{c}_i(\mathbb{K})\in\mathcal{F}$. Thus $\mathbb{K}\in\Phi(\mathcal{F})$.

\item Moreover, for any arrow $\alpha \mathcolon k \rightarrow i$ with $i \neq k$, we have the commutative diagram
\[
\xymatrix{
0 \ar[r] & \mathbb{K}_{k} = \ker(\sigma)\ar[d] \ar[r]^l & \mathbb{P}_{k} = P \ar[d]^{\mathbb{P}_\alpha} \ar[r]^\sigma & \mathbb{F}_{k} = F \ar[d] \ar[r] & 0 \\
0 \ar[r] & \mathbb{K}_i \ar[r]^(.4){=} & \mathbb{P}_i = \bigoplus \limits_{Q(k,i)} P \ar[r] & \mathbb{F}_i = 0 \ar[r] & 0
}
\]
where $l$ and $\mathbb{P}_\alpha$ are canonical injections.
\end{itemize}

Let us prove now the claim (2). So let $G\in {^{\perp_1}\mathcal{L}}$. We want to show that $G\in\mathcal{F}^{\perp_1}$. Given $F\in\mathcal{F}$, we have the previous exact sequence
\[
0\to \ker(\sigma)\stackrel{l}{\to} P\stackrel{\sigma}{\to} F\to 0,
\]
with $P$ projective. Then to get what we claim, it suffices to show that any $f\mathcolon \ker(\sigma)\to G$ can be lifted to a map $P\to G$, that is, the previous sequence is $\Hom_{\mathcal C}(-,G)$ exact. So, let $f\mathcolon \ker(\sigma)\rightarrow G$  be any morphism and let $\widetilde{f}\mathcolon\mathbb{K}\rightarrow\mathbbm{s}_k(G)$ be the induced morphism in $\Rep(Q,\mathcal C)$ with $\widetilde{f}_{ki} = \delta_{ij}f.$ Note that, since $G\in  {^{\perp_1}\mathcal{L}}$, we get from Corollary \ref{cor 1} and the hypothesis on the balance that $\widetilde{id}\mathcolon e_\lambda^k(G)\rightarrow\mathbbm{s}_k(G)\rightarrow0$ is $\Hom_{\mathcal C}(\Phi(\mathcal{F}),-)$ exact. And we have previously proved that $\mathbb{K}\in \Phi(\mathcal F)$. Therefore, for the map $\widetilde{f}\mathcolon\mathbb{K}\rightarrow\mathbbm{s}_k(G)$,  there is $\widetilde{g}:\mathbb{K}\rightarrow e_\lambda^k(G)$ such that $\widetilde{f}=\widetilde{id}\widetilde{g}$. In particular, for the arrow $\alpha\mathcolon k\rightarrow i$, we have the following commutative diagram
\[
\begin{tikzpicture}[description/.style={fill=white,inner sep=2pt}]
\matrix (m) [matrix of math nodes, row sep=2em, column sep=2em, text height=1.5ex, text depth=0.25ex]
{
{} & {} & {\rm Ker}(\sigma) & {} & {} \\
{} & G & {} & G & 0 \\
{} & {} & \mathbb{K}_i & {} & {} \\
G & \bigoplus_{Q(k,i)} G & {} & 0 & 0 \\
};
\path[-latex]
(m-1-3) edge node[below,sloped] {\footnotesize$\hspace{-1.25cm}\mathbb{P}_\alpha l$} (m-3-3)
(m-2-2) edge [-,line width=6pt,draw=white] (m-2-4) edge [,->] (m-2-4)
(m-1-3) edge node[above,sloped] {\footnotesize$\tilde{g}_k$} (m-2-2)
(m-1-3) edge node[above,sloped] {\footnotesize$f$} (m-2-4)
(m-2-4) edge (m-2-5)
(m-2-2) edge node[left] {\footnotesize$e^k_{\lambda}(G)_\alpha$} (m-4-2)
(m-2-4) edge (m-4-4)
(m-3-3) edge node[above,sloped] {\footnotesize$\tilde{g}_i$} (m-4-2)
(m-4-2) edge node[below] {\footnotesize$\pi_\alpha$} (m-4-1)
(m-4-2) edge (m-4-4)
(m-4-4) edge (m-4-5)
(m-3-3) edge (m-4-4)
;
\end{tikzpicture}
\]
It follows that $\widetilde{g}_i\mathbb{P}_\alpha l = {e_\lambda^k(G)}_\alpha\widetilde{g}_k$. Let $\pi_{\alpha}$ be the canonical projection corresponding to the canonical injection ${e_\lambda^k(G)}_\alpha$, and so
\[
f = \widetilde{g}_k=\pi_{\alpha}{e_\lambda^k(G)}_\alpha\widetilde{g}_k=(\pi_{\alpha}\widetilde{g}_i)\circ (\mathbb{P}_\alpha l)=(\pi_{\alpha}\widetilde{g}_i \mathbb{P}_\alpha)\circ  l.
\]
That is, the sequence $0 \rightarrow \ker(\sigma) \rightarrow P\rightarrow F\rightarrow 0$ is $\Hom_{\mathcal C}(-,G)$ exact,  and so $G\in\mathcal{F}^{\perp_1}$.
\end{enumerate}
\epf

For the following results, recall (see e.g. \cite{HJ}) that a quiver $Q$ is said to be \emph{left rooted} if it contains no paths of the form $\cdots \to \bullet \to \bullet \to \bullet$. Dually, $Q$ is called \emph{right rooted} if it contains no paths of the form $\bullet \to \bullet \to \bullet \to \cdots$.

Let us focus now in the case $\mathcal C=\mathsf{Mod}(R)$. If the quiver $Q$ is \emph{left and right rooted} (for instance the quiver $\cdots \rightarrow \bullet \leftarrow \bullet \rightarrow \bullet \leftarrow \bullet \rightarrow \bullet \leftarrow \bullet \rightarrow \cdots$) we can combine \cite[Theorems A and B]{HJ} and \cite[Theorem A]{EHHS} to infer that, in case we start with two complete hereditary cotorsion pairs $(\mathcal{F},\mathcal{H})$ and $(\mathcal{G},\mathcal{L})$ in $\mathsf{Mod}(R)$, then we get two induced complete hereditary cotorsion pairs $(\Phi(\mathcal{F}), \Rep(Q,\mathcal{H}))$ and $(\Rep(Q,\mathcal{G}), \Psi(\mathcal{L}))$ in $\Rep(Q,\mathsf{Mod}(R))$. Therefore, from the previous proposition and Corollary~\ref{cor.b.pairs.cot.triplets} we get the following result.

\begin{cor}\label{cor 2}
If $(\mathcal{F},\mathcal{H})$ and $(\mathcal{G},\mathcal{L})$ are complete hereditary cotorsion pairs in $\mathsf{Mod}(R)$, then the following are equivalent:
\begin{itemize}
\item [(a)]$\mathcal{H}=\mathcal{G}$

\item [(b)] $(\Phi(\mathcal{F}),\Psi(\mathcal{L}))$ is a balanced pair for any non-discrete left and right rooted quiver $Q$.

\item [(c)] $(\Phi(\mathcal{F}),\Psi(\mathcal{L}))$ is a balanced pair for some non-discrete left and right rooted quiver $Q$.
\end{itemize}
\end{cor}
%In case the quiver consists of just one vertex and no arrows, we get the following remarkable consequence of the previous Corollary.
%\begin{cor}\label{cor 3} If $(\mathcal{F},\mathcal{H})$ and $(\mathcal{G},\mathcal{L})$ are complete hereditary cotorsion pairs, then the following are equivalent:
%\begin{enumerate}
%\item$\mathcal{H}=\mathcal{G}$, that is, $(\mathcal{F,G,L})$ is a hereditary and complete cotorsion triplet in $\mathcal C$.
%\item $(\mathcal{F},\mathcal{L})$ is a balanced pair in $\mathcal C$.
%%\item $(\Phi(\mathcal{F}),\Psi(\mathcal{L}))$ is a balanced pair for a quiver $Q$.
%\end{enumerate}\end{cor}

As a consequence of Corollary \ref{cor 2}, Example~\ref{ex:triplets}~\eqref{ex:quasiFrobenius} and Remark~\ref{rem:quasiFrobenius}, we have the following characterization of quasi-Frobenius rings.

\begin{cor}\label{cor:qf}
A ring $R$ is quasi-Frobenius if and only if $(\Phi(\mathsf{Mod}(R)), \Psi(\mathsf{Mod}(R)))$ is a balanced pair for a non-discrete left and right rooted quiver $Q$. In this case, we have the complete hereditary cotorsion triplet $(\Phi(\mathsf{Mod}(R)),{\rm Rep}(Q,\mathsf{Proj}(R)),\Psi(\mathsf{Mod}(R)))$ in ${\rm Rep}(Q,\mathsf{Mod}(R))$.
\end{cor}

\begin{remark}
The category $\Phi(\mathsf{Mod}(R))$ is known in the literature as \emph{monomorphism category}. It has been extensively studied by Li, Luo and Zhang in \cite{LiZ,LuoZ}. Dually, $\Psi(\mathsf{Mod}(R))$ is called \emph{epimorphism category}.
\end{remark}

Our last result allows to give another extension of the characterization of virtually Gorenstein Noetherian rings of finite Krull dimension given by Zareh-Khoshchehreh, Asgharzadeh and Divaani-Aazar in \cite[Theorem 3.10]{ZAD}. We recall that a ring $R$ is called \emph{left $n$-perfect} if every flat left $R$-module has finite projective dimension $\leq n$.

\begin{cor}\label{cor:ZAD2}
Let $R$ be a left $n$-perfect and right coherent ring. Then, the following conditions are equivalent:
\begin{itemize}
\item[(a)] $R$ is virtually Gorenstein.

\item[(b)] $(\Phi(\mathsf{GProj}(R)),\Psi(\mathsf{GInj}(R)))$ is a balanced pair in $\Rep(Q,\mathsf{Mod}(R))$ for some non-discrete left and right rooted quiver $Q$.

\item[(c)] $(\mathsf{GProj}(R),\mathsf{GInj}(R))$ is a balanced pair in $\mathsf{Mod}(R)$.
\end{itemize}
\end{cor}

\pf
Firstly we point out that under the assumptions on $R$, the pair $(\mathsf{GProj}(R),\mathsf{GProj}(R)^{\perp})$ is known to be a complete hereditary cotorsion pair (see Estrada, Iacob, Odaba\c s\i\ \cite[Proposition 6]{EIO}). On the other hand, \v{S}aroch and  \v{S}\v{t}ov{\'{\i}}{\v{c}}ek (\cite{SS}) have recently proved that the pair $({^{\perp}\mathsf{GInj}(R)}, \mathsf{GInj}(R))$ is a perfect (so, in particular, complete) and hereditary cotorsion pair for \emph{any} ring.

Now, (a) $\Leftrightarrow$ (c) immediately follows from Corollary~\ref{cor.b.pairs.cot.triplets} by the above and by noticing that
\[
\mathsf{GProj}(R) \cap \mathsf{GProj}(R)^\perp = \mathsf{Proj}(R) \mbox{ \ and \ } {}^{\perp}\mathsf{GInj}(R) \cap \mathsf{GInj}(R) = \mathsf{Inj}(R).
\]
Finally (a) $\Leftrightarrow$ (b) follows from Corollary~\ref{cor 2}.
\epf

%%%%%%%%%%%%%%%%%%%%%%%%%%%%%%%%%%%%%%%%%%%%%%%%%%
%%%%%%%%%%%%%%%%%%%%%%%%%%%%%%%%%%%%%%%%%%%%%%%%%%
%%%%%%%%%%%%%%%%%%ACKNOWLEDGEMENTS%%%%%%%%%%%%%%%%%%%
%%%%%%%%%%%%%%%%%%%%%%%%%%%%%%%%%%%%%%%%%%%%%%%%%%
%%%%%%%%%%%%%%%%%%%%%%%%%%%%%%%%%%%%%%%%%%%%%%%%%%

\section*{Acknowledgements}

The authors wish to thank Jiansheng Hu for useful comments and suggestions during the preparation of this manuscript.
This research was partially supported by NSFC (11671069), the Zhejiang Provincial Natural Science Foundation of China(LY18A010032).

%%%%%%%%%%%%%%%%%%%%%%%%%%%%%%%%%%%%%%%%%%%%%%%%%%
%%%%%%%%%%%%%%%%%%%%%%%%%%%%%%%%%%%%%%%%%%%%%%%%%%
%%%%%%%%%%%%%%%%%%%%%REFERENCES%%%%%%%%%%%%%%%%%%%%%
%%%%%%%%%%%%%%%%%%%%%%%%%%%%%%%%%%%%%%%%%%%%%%%%%%
%%%%%%%%%%%%%%%%%%%%%%%%%%%%%%%%%%%%%%%%%%%%%%%%%%

\bibliographystyle{plain}
\bibliography{references}

\begin{thebibliography}{10}

\bibitem{BR}
A.~{Beligiannis} and I.~{Reiten}.
\newblock {Homological and homotopical aspects of torsion theories.}
\newblock {\em {Mem. Am. Math. Soc.}}, 883:207, 2007.

\bibitem{Chen}
X.-W. {Chen}.
\newblock {Homotopy equivalences induced by balanced pairs.}
\newblock {\em {J. Algebra}}, 324(10):2718--2731, 2010.

\bibitem{DC}
N.~Ding and J.~Chen.
\newblock The flat dimensions of injective modules.
\newblock {\em Manuscripta Math.}, 78(2):165--177, 1993.

\bibitem{Enochs}
E.~E. {Enochs}.
\newblock {Balance with flat objects.}
\newblock {\em {J. Pure Appl. Algebra}}, 219(3):488--493, 2015.

\bibitem{EEGI}
E.~E. {Enochs}, S.~{Estrada}, J.R. {Garc{\'\i}a Rozas}, and A.~{Iacob}.
\newblock {Gorenstein quivers.}
\newblock {\em {Arch. Math.}}, 88(3):199--206, 2007.

\bibitem{EH}
E.~E. {Enochs} and I.~{Herzog}.
\newblock {A homotopy of quiver morphisms with applications to
  representations.}
\newblock {\em {Can. J. Math.}}, 51(2):294--308, 1999.

\bibitem{EJ}
E.~E. {Enochs} and O.~M.~G. {Jenda}.
\newblock {\em {Relative {H}omological {A}lgebra. Vol. 1. 2nd revised and
  extended ed.}}
\newblock Berlin: Walter de Gruyter, 2nd revised and extended ed. edition,
  2011.

\bibitem{EJ2}
E.~E. {Enochs} and O.~M.~G. {Jenda}.
\newblock {\em {Relative {H}omological {A}lgebra. Vol. 2. 2nd revised ed.}}
\newblock Berlin: Walter de Gruyter, 2nd revised ed. edition, 2011.

\bibitem{EJTX}
E.~E. {Enochs}, O.~M.~G. {Jenda}, B.~{Torrecillas}, and J.~{Xu}.
\newblock Torsion theory relative to {E}xt, 1998.

\bibitem{EHHS}
H.~{Eshraghi}, R.~{Hafezi}, E.~{Hosseini}, and S.~{Salarian}.
\newblock {Cotorsion theory in the category of quiver representations.}
\newblock {\em {J. Algebra Appl.}}, 12(6):1350005, 16, 2013.

\bibitem{EIO}
S.~{Estrada}, A.~{Iacob}, and S.~{Odaba\c s\i}.
\newblock Gorenstein flat and projective and precovers.
\newblock {\em {Pub. Math. Debrecen}}, 91(1-2(7)), 2017.

\bibitem{G}
J.~{Gillespie}.
\newblock {Kaplansky classes and derived categories.}
\newblock {\em {Math. Z.}}, 257(4):811--843, 2007.

\bibitem{GillespieDegree}
J.~Gillespie.
\newblock Cotorsion pairs and degreewise homological model structures.
\newblock {\em Homology Homotopy Appl.}, 10(1):283--304, 2008.

\bibitem{Gi}
J.~{Gillespie}.
\newblock {Model structures on modules over {D}ing-{C}hen rings.}
\newblock {\em {Homology Homotopy Appl.}}, 12(1):61--73, 2010.

\bibitem{GH}
P.~A. {Guil Asensio} and I.~{Herzog}.
\newblock {Sigma-cotorsion rings.}
\newblock {\em {Adv. Math.}}, 191(1):11--28, 2005.

\bibitem{HJ}
H.~{Holm} and P.~{J\o rgensen}.
\newblock Cotorsion pairs in categories of quiver representations.
\newblock {\em {To appear in Kyoto J. Math.}}, page 23 pp.

\bibitem{Hovey}
M.~{Hovey}.
\newblock {Cotorsion pairs, model category structures, and representation
  theory.}
\newblock {\em {Math. Z.}}, 241(3):553--592, 2002.

\bibitem{Krause}
H.~Krause.
\newblock The stable derived category of a {N}oetherian scheme.
\newblock {\em Compos. Math.}, 141(5):1128--1162, 2005.

\bibitem{LiZ}
Z.-W. {Li} and P.~{Zhang}.
\newblock {A construction of {G}orenstein-projective modules.}
\newblock {\em {J. Algebra}}, 323(6):1802--1812, 2010.

\bibitem{LuoZ}
X.-H. {Luo} and P.~{Zhang}.
\newblock {Monic representations and {G}orenstein-projective modules.}
\newblock {\em {Pac. J. Math.}}, 264(1):163--194, 2013.

\bibitem{M}
B.~{Mitchell}.
\newblock {Rings with several objects.}
\newblock {\em {Adv. Math.}}, 8:1--161, 1972.

\bibitem{SS}
J.~{\v{S}}aroch and J.~{\v{S}}\v{t}ov{\'{\i}}{\v{c}}ek.
\newblock Singular compactness for $\sigma$-cotorsion and projectively resolved
  gorenstein flat modules.
\newblock 2017.

\bibitem{ZAD}
F.~{Zareh-Khoshchehreh}, M.~{Asgharzadeh}, and K.~{Divaani-Aazar}.
\newblock {{G}orenstein homology, relative pure homology and virtually
  {G}orenstein rings.}
\newblock {\em {J. Pure Appl. Algebra}}, 218(12):2356--2366, 2014.

\end{thebibliography}

\end{document}